\documentclass[aap,MSNbibl,seceqn,rotating,dvips]{arximspdf}
\usepackage{mathbh}
\usepackage{graphicx}

\doi{10.1214/12-AAP883} 
\volume{23}
\issue{4}
\pubyear{2013}
\firstpage{1660}
\lastpage{1691}

\makeatletter

\renewcommand{\mid}{|}

\newcommand{\lbra}{[\![}
\newcommand{\rbra}{]\!]}

\newcommand{\one}{{\mathbh1}}

\newtheorem{proposition}{Proposition}[section]
\newtheorem{theorem}{Theorem}[section]
\newproclaim{definition}{Definition}[section]
\newtheorem{corollary}{Corollary}[section]
\newtheorem{lemma}{Lemma}[section]
\newproclaim{remark}{Remark}[section]

\makeatother

\begin{document}
\begin{frontmatter}

\title{Weak approximations for Wiener functionals}
\runtitle{Weak approximations for Wiener functionals}

\begin{aug}
\author[A]{\fnms{Dorival} \snm{Le\~{a}o}\ead[label=e1]{leao@icmc.usp.br}}
\and
\author[B]{\fnms{Alberto} \snm{Ohashi}\corref{}\thanksref{t3}\ead[label=e3]{ohashi@mat.ufpb.br}}
\runauthor{D. Le\~{a}o and A. Ohashi}
\affiliation{Universidade de S\~{a}o Paulo, and Universidade Federal da Paraiba and~Insper~Institute}
\address[A]{Departamento de Matem\'atica Aplicada\\
\quad e Estat\'istica\\
Universidade de S\~{a}o Paulo\\
13560-970 S\~ao Carlos, SP\\
Brazil\\
\printead{e1}}
\address[B]{Departamento de Matematica\\
Universidade Federal da Paraiba\\
58051-900, Joao Pessoa\\
Brazil\\
\printead{e3}} 
\end{aug}

\thankstext{t3}{Supported by CNPq Grant 308742.}

\received{\smonth{3} \syear{2012}}

%
\begin{abstract}
In this paper we introduce a simple space-filtration discretization
scheme on Wiener space which allows us to study weak decompositions and
smooth explicit approximations for a large class of Wiener functionals.
We show that any Wiener functional has an underlying robust
semimartingale skeleton which under mild conditions converges to it.
The discretization is given in terms of discrete-jumping filtrations
which allow us to approximate nonsmooth processes by means of a
stochastic derivative operator on the Wiener space. As a by-product, we
provide a robust semimartingale approximation for weak Dirichlet-type
processes.

The underlying semimartingale skeleton is intrinsically constructed in
such way that all the relevant structure is amenable to a robust
numerical scheme. In order to illustrate the results, we provide an
easily implementable approximation scheme for the classical
Clark--Ocone formula in full generality. Unlike in previous works, our
methodology does not assume an underlying Markovian structure and does
not require Malliavin weights. We conclude by proposing a method that
enables us to compute optimal stopping times for possibly non-Markovian
systems arising, for example, from the fractional Brownian motion.
\end{abstract}

%

\end{frontmatter}

\section{Introduction}\label{sec1}
Discretization methods for stochastic systems have always been a topic
of great interest
in stochastic analysis and its applications. Since the pioneering work
of Wong and Zakai we
know that not every choice of discretization procedure leads to good
stability properties
of elementary processes such as It\^o integrals and related stochastic
equations. See, for example, the works
\cite{slominski,protter,jacod1,Barlow1,coquet3} and other references therein.

In order to get those convergence results, one has to assume suitable
compactness arguments which allow
one to exchange the limits. On the one hand, one may interpret such
assumptions as simple technical arguments
imposed on the system to get the desirable robustness. On the other
hand, Graversen and Rao~\cite{gr}
have shown a close relation between finite energy and the existence of
Doob--Meyer-type decompositions. More recently, Coquet et~al.
\cite{coquet2} has proved the uniqueness of such decompositions by means of
the so-called weak Dirichlet processes.

The classical Graversen--Rao theorem can be proved by means of
compactness arguments on predictable compensators of simple
time-discretizations of the original process. In general, approximating
sequences arising from such compactness arguments are not intrinsically
constructed and are not suitable for numerical schemes in nonstandard
cases arising from non-Markovian and nonsemimartingale systems.

The primary goal of this work is to describe readable structural
conditions on a given optional process adapted to the Brownian
filtration (henceforth abbreviated by Wiener functional) so that one
can construct an \textit{explicit}, \textit{robust} and \textit
{feasible} approximating skeleton of smooth semimartingales. In order
to illustrate the basic idea, let us assume that a Wiener functional
$X$ has an abstract representation
%
\begin{equation}
\label{wienerfun} X_t=X_0 + \int_0^t
H_s\,dB_s + N_t,
\end{equation}
where $B$ is the standard Brownian motion under its natural filtration
$\mathbb{F}$, $N$ can be a nonsemimartingale $\mathbb{F}$-optional
process and $H$ is a progressive process which is completely unknown a
priori. The main problem addressed in this paper is the following one:
Construct an explicit and simple sequence of $\mathbb{F}^k$-special
semimartingales given by
\[
X^k = X_0 + \int H^k\,dA^k +
N^k, \qquad\mathbb{F}^k\subset\mathbb{F},
\]
where $H^k$ is fully based on the information generated by the pair
$(X, B)$ such that
%
\begin{eqnarray}
\label{qq} H^k&\rightarrow& H,\qquad A^k\rightarrow B,\qquad \int
H^ k\,dA^ k\rightarrow\int H\,dB,\nonumber\\[-8pt]\\[-8pt]
N^k&\rightarrow& N,\qquad
\mathbb{F}^k\rightarrow\mathbb{F}\qquad\mbox{as }
k\rightarrow\infty.\nonumber
\end{eqnarray}

The main difficulty in answering this question comes from the fact that
when $X$ is very rough, the joint convergence of $(H^k, \int H^k\,dA^k)$
to $(H, \int H\,dB)$ in general will not hold since $H$ has no a priori
path regularity. Similarly, $N$ can be very irregular in such a way
that $N^k\rightarrow N$ will not hold either. A~similar type of problem
was addressed by Jacod, Meleard and Protter~\cite{jacod1} in a pure
martingale and Markovian setup at a fixed terminal time $0 < T \le
\infty$. In
\cite{jacod1}, they have provided reasonable explicit expressions for
$H^k$ when $B$ and $N$ are replaced by orthogonal square-integrable
martingales w.r.t. an arbitrary filtration. More explicit expressions\vadjust{\goodbreak}
were obtained by imposing an underlying Markovian structure. In this
paper, we are interested in somehow more irregular objects arising from
non-Markovian and nonsemimartingale systems restricted to the Wiener space.

In order to study Wiener functionals of type (\ref{wienerfun}), an
abstract theory is developed based on an underlying smooth
semimartingale skeleton induced by a suitable sequence of stopping
times which measures the instants when the Brownian motion hits some a
priori levels. More precisely, from a given Brownian motion $B$ we
shall define inductively a sequence of stopping times
\[
T^k_n:= \inf\bigl\{ T^k_{n-1}< t <
\infty; | B_t - B_{T^k_{n-1}}| = 2^{-k}\bigr\},\qquad n \ge1,
\]
which induces an embedded semimartingale structure of the form
\[
\delta^k X_t:= X_0 + \sum
_{n=1}^\infty\mathbb{E}\bigl[X_{T^k_n}|\mathcal
{G}^k_n\bigr] \one_{\{ T_{n}^{k} \leq t <
T^k_{n+1} \}},\qquad 0\le t \le T,
\]
for a suitable family $(\mathcal{G}^k)$ of discrete-time filtrations.
By the very definition, $\delta^kX$ should be interpreted as a
space-filtration discretization scheme.

In this work, we prove that under mild conditions which are similar in
nature to weak Dirichlet-type processes, $\delta^k X$ induces a robust
skeleton $(A^k,\int H^k\,dA^k, N^k, \mathbb{F}^k)$ which realizes (\ref
{qq}) in suitable topologies. Beyond that, and more importantly for
applications, the skeleton is amenable to a feasible numerical analysis
by means of perfect simulations of the first-passage times of the
Brownian motion (see Burq and Jones~\cite{burq}).

The second part of this article is devoted to the application of our
abstract results to the pure martingale case. To illustrate the
techniques developed in this paper, we present a step-by-step
simulation method for the Clark--Ocone formula in full generality.
Recall that if $Y\in L^2(\mathcal{F}_T)$, then
\[
Y=\mathbb{E}[Y] + \int_0^T
\mathbb{E}[D_sY|\mathcal{F}_s]\,dB_s,
\]
where $D$ stands for the Gross--Sobolev derivative on the
Gaussian space of the Brownian motion. The process
$\mathbb{E}[DY|\mathcal{F}]$ has great importance in mathematical
finance because it is the fundamental quantity for the hedging problem
in a complete Brownian-based market (see, e.g.,~\cite{karatzas}).
However, the practical implementation of the Clark--Ocone formula is
still an open problem mainly because $D_tY$ is only amenable to
numerical schemes in very particular cases such as elliptic systems
where the Malliavin weights can be efficiently used. See, for example,
\cite{FOURNIE1,higa,elie,ben} for a complete discussion on this matter.

In this article, we propose a rather different approach based on the
sequence of stochastic ratios
%
\begin{equation}
\label{optionalpr} \frac{\mathbb{E}[Y|\mathcal{G}^k_n]- \mathbb{E}[Y|\mathcal
{G}^k_{n-1}]}{B_{T^k_n}- B_{T^k_{n-1}}};\qquad k,n\ge1.
\end{equation}

Unlike in previous works (see, e.g.,~\cite{dinunno,dinunno1}), the
approximation scheme given in (\ref{optionalpr}) is intrinsic and it is
rather explicit without imposing smoothness in the sense of Malliavin
calculus and no underlying Markovian structure is assumed (see also,
e.g.,~\cite{aase,FOURNIE1}). Moreover, no functional pathwise
smoothness is required in the approximation of $\mathbb{E}[DY|\mathcal
{F}]$ (see Dupire~\cite{dupire} and Cont and Fournie~\cite{cont} for
some results in this direction). More importantly for applications,
$\mathbb{E}[DY|\mathcal{F}]$ is the limit of functionals of (\ref
{optionalpr}) which are fully described by the sequences of smooth
i.i.d. stopping times $(T^k_n-T^k_{n-1})_{n\ge1}$ and the Bernoulli
variables $(B_{T^k_n}-B_{T^k_{n-1}})_{n\ge1}$. This makes our
approximation explicit and easily implementable for a very large class
of payoffs. Based on (\ref{optionalpr}), we present a step-by-step
simulation method for the Clark--Ocone formula. To the best of our
knowledge, the proposed methodology is the only one capable of
simulating $\mathbb{E}[DY|\mathcal{F}]$ for arbitrary square-integrable
$\mathcal{F}_T$-random variables.

In the last part of the article, we illustrate our discretization
scheme with optimal stopping problems arising in non-Markovian systems.
We propose an algorithm fully based on our discretization scheme which
allows us to simulate value functions and the optimal stopping times
for continuous Wiener functionals arising in genuinely non-Markovian
cases such as, for example, the fractional Brownian motion.

The remainder of the article is structured as follows. In Section~\ref{sec2}, we
fix the notation and we give some preliminary results regarding the
pre-limit sequence and its basic properties. In Section~\ref{sec3}, we establish
the convergence of the semimartingale skeleton. Section~\ref{sec4} is devoted to
the stochastic derivative. In Section~\ref{sec5}, a step-by-step algorithm to
simulate the Clark--Ocone formula is presented. Section~\ref{sec6} presents an
optimal stopping time algorithm based on the discretization scheme
developed in this article.

\section{Preliminaries}\label{sec2}
In this section we fix the basic notation and framework that we
use in this paper and present some elementary results concerning
our approximation scheme. Throughout this paper we are given the
usual stochastic basis $(\Omega, \mathbb{F},
\mathcal{F},\mathbb{P})$ of the standard Brownian motion $B$ starting from
$0$, where $\Omega$ is the set $\mathcal{C}(\mathbb{R}_+;
\mathbb{R}):= \{f\dvtx \mathbb{R}_+\rightarrow
\mathbb{R} \mbox{ continuous}; f(0)=0\}$, $\mathcal{F}$ is the
completed Borel sigma algebra, $\mathbb{P}$ is the Wiener measure
on $\Omega$ and $\mathbb{F}:=(\mathcal{F}_t)_{t\ge0}$ is the
usual $\mathbb{P}$-augmentation of the natural filtration
generated by the Brownian motion. We denote by $\mathcal{O}$ the
optional sigma algebra with respect to
$\mathbb{F}$.

For each positive integer $k$, we define $T^k_0 = 0$ a.s. and
%
\begin{equation}
\label{stoppingtimes} T^k_n:= \inf\bigl\{
T^k_{n-1}< t <\infty; | B_t - B_{T^k_{n-1}}|
= 2^{-k}\bigr\},\qquad n \ge1.
\end{equation}
One should notice that $(T^k_n)_{n\ge0}$ is an exhaustive sequence of
$\mathbb{F}$-stopping times for every $k$ where $\{T^k_{n} -
T^k_{n-1}\}_{n=1}^\infty$ is an i.i.d. sequence. Next we consider the
following family of random variables:
%
\begin{equation}
\label{sigmakn} \sigma^k_n:=\cases{ 1; &\quad if
$B_{T^k_n} - B_{T^k_{n-1}} = 2^{-k}$ and $T^k_n
< \infty$,
\cr
-1; &\quad if $B_{T^k_n} - B_{T^k_{n-1}} = -2^{-k}$
and $T^k_n < \infty$,
\cr
0; &\quad if $T^k_n
= \infty$.}
\end{equation}
We then define the following sequence of step processes as
\[
A^k_t:= \sum_{n=1}^{\infty}2^{-k}
\sigma^k_n\one_{\{ T_{n}^{k}
\leq t \}},\qquad 0\le t < \infty; k\ge1.
\]
For each $k\ge1$, let $(\mathcal{F}^k_t)_{t\ge0}$ be the
natural filtration generated by $\{A^k_t; 0\le t < \infty\}$. One
should notice that $(\mathcal{F}^k_t)_{t\ge0}$ is a discrete-type
filtration (see, e.g.,~\cite{yan}, page~321) in the sense that
%
\begin{equation}
\label{discrete} \mathcal{F}^k_t = \bigcup
_{i=0}^{\infty} \bigl( \mathcal{G}^k_i
\cap\bigl\{T^k_i \le t < T^k_{i+1}
\bigr\} \bigr),\qquad t \ge0,
\end{equation}
where $\mathcal{G}^k_0:= \{\Omega, \varnothing\}$ and
$\mathcal{G}^k_n:=\mathcal{F}^k_{T^k_n}=\sigma(T^k_1, \ldots,
T^k_n, \sigma^k_1, \ldots, \sigma^k_n)$. Moreover, since
$\mathcal{G}^k_n = \sigma(A^k_{s\wedge T^k_n}; s \ge0)$ then
$\mathcal{G}^k_n$ and $\mathcal{F}^k_t$ coincide up to
$\mathbb{P}$-null sets on $ \{T^k_n \le t < T^k_{n+1} \}$.
In other words, $(\mathcal{F}_t^k)_{t\ge0}$ is a jumping filtration
(e.g.,~\cite{jacod})
with jumping sequence given by $(T^k_n)_{n\ge1}$ for each $k\ge
1$. With a slight abuse of notation we write $\mathcal{F}^k_t$ to
denote its $\mathbb{P}$-augmentation satisfying the usual
conditions, where $\mathbb{F}^k:= (\mathcal{F}^k_t)_{t\ge0}$. We
also denote by $\mathcal{O}^k$ and $\mathcal{P}^k$ the optional
and predictable sigma algebras, respectively, with respect to
$\mathbb{F}^k$.

In this work, the $\mathbb{F}^k$-dual predictable and optional
projections of a real-valued measurable process $Y$ will be denoted by
$[Y]^{p,k}$ and $[Y]^{o,k}$, respectively. We also denote by $[X,Y]$
and $\langle X,Y\rangle$ the usual quadratic variation and predictable
bracket of a pair of semimartingales, respectively. The usual jump of a
process is denoted by $\Delta Y_t
= Y_t - Y_{t-}$ where $Y_{t-}$ is the left-hand limit of a c\`adl\`ag
process $Y$. We set $Y_{0-}=Y_0$ for convenience. Moreover, if $T$ and
$S$ are stopping times, then $\lbra T,S\rbra $, $\lbra T,S\lbra $ and $\rbra T,S\rbra $ will
denote the usual stochastic intervals. From now on we fix a terminal
time $ 0 < T < \infty$.

We now give some elementary properties of our discretization scheme.

\begin{lemma}\label{useful}
For each $k\ge1$, $\{A^k_t; 0 \le t \le T \}$ is an $\mathbb
{F}^k$-martingale with locally
integrable variation such that
%
\begin{equation}
\label{approximationbm}
\sup_{0\le t \le T} \bigl\|B_t - A^k_t \bigr\|_{\infty} \le2^{-k},
\end{equation}
where $\|\cdot\|_{\infty}$ denotes the usual norm on the
space $L^{\infty}(\mathbb{P})$. Moreover, $\mathbb{F}^k$ is a quasi
left-continuous filtration and it supports only martingales of bounded
variation.
\end{lemma}
\begin{pf}
The estimate (\ref{approximationbm}) and the locally integrable
variation property are
immediate consequences of the definitions.\vadjust{\goodbreak} For the martingale
property we notice from (\ref{discrete}) that we can write
\[
\mathcal{F}^k_t = \Biggl\{ \bigcup
_{n=0}^{\infty} A_n \cap\bigl[T^k_n
\le t < T^k_{n+1}\bigr]; A_n \in
\mathcal{G}^k_n, n\ge0 \Biggr\},\qquad t\ge0,
\]
where $A^k_s = B_{T^k_n}$ on $[T^k_n \le s < T^k_{n+1}]$ for
each $n\ge1$. In this case, the usual
optional stopping theorem gives the representation (see also Remark \ref
{martingaleprojectionremark})
\[
\mathbb{E}\bigl[B_T|\mathcal{F}^k_t\bigr] =
A^k_t\qquad\mbox{a.s.},\qquad 0\le t\le T,
\]
and therefore we may conclude that $A^k$ is an
$\mathbb{F}^k$-martingale. For the second part, we notice that since
$T^k_1$ is an absolutely continuous random variable and $A^k$ is a
point process, then in this case it is well known that $\mathbb{F}^k$
is a quasi left-continuous filtration. The fact that every $\mathbb
{F}^k$-martingale has bounded variation is a consequence of~\cite{jacod}.
\end{pf}

In the sequel, we denote by $\pi$ the usual projection of $\mathbb
{R}_+\times\Omega$ onto $\Omega$. For any measurable sets $D$ and $A$
we write $ D - A$ to denote $D \cap A^c $, where $A^c$ is the
complement of the set $A$. Moreover, $\bigvee_{k\ge0}\mathcal{A}_k$
denotes the\vspace*{1pt} sigma-algebra generated by
$\bigcup_{k\ge0}\mathcal{A}_k$ for a sequence of classes
$\{\mathcal{A}_k; k\ge0\}$.
%
\begin{lemma}\label{smoothfiltration}
The natural filtration of $A^k$ satisfies the following properties:

\begin{longlist}[(iii)]
\item[(i)] $\{\mathbb{F}^k;k\ge1\}$ is\vspace*{1pt} an increasing family of
sigma-algebras such that $\mathcal{F}_t = \bigvee_{k\ge
0}\mathcal{F}^k_t$ for every $t \ge0$.

\item[(ii)] The sequence of filtrations $\mathbb{F}^k$ converges
weakly to $\mathbb{F}$.

\item[(iii)] For every $O\in\mathcal{O}$ there exists a
sequence $O^k\in\mathcal{O}^k$ such that
\[
O^k\subset O\qquad \forall k\ge1 \quad\mbox{and}\quad \mathbb{P} \bigl[\pi(O)
- \pi \bigl(O^k\bigr)\bigr]\rightarrow0 \qquad\mbox{as } k\rightarrow\infty.
\]
\end{longlist}
\end{lemma}
\begin{pf}
It is straightforward to check that $\mathcal{F}^k_t\subset\mathcal
{F}_t^{k+1}$ for every $k$ and $t \ge0$.
Moreover, each cylinder set of the form $\{b_1 < B_t \le b_2\}$ can be
approximated by
%
\begin{eqnarray}
\label{f1}
&&
\bigl\{b_1 +2^{-k}< A^k_t
\le b_2 - 2^{-k} \bigr\} \nonumber\\
&&\qquad
\subset\{b_1 <
B_t \le b_2 \} \\
&&\qquad\subset\bigl\{b_1 -
2^{-k}< A^k_t \le b_2 +
2^{-k}\bigr\} \qquad\mbox{ a.s.}\nonumber
\end{eqnarray}
for $k$ large enough, thus proving part (i). To prove part (ii) we only
need to show that for each $B \in\mathcal{F}_T$ the sequence of martingales
$\mathbb{E}[\one_B|\mathcal{F}^k_{\cdot}]$ converges in probability to
$\mathbb{E}[\one_B|\mathcal{F}_{\cdot}]$ on the space of c\`adl\`ag
functions equipped with the usual Skorohod topology. But this is a
simple application of~\cite{coquet1}, Proposition~4. Now let us fix an
arbitrary $0 < t \le T$. From (\ref{f1}) we know that for any cylinder
set restricted on $[0,t]$ we may find two sequences $(D^k_{i})_{k\ge
1}$, $i=1,2$, such that
%
\begin{equation}
\label{e3} D^k_1 \subset D \subset D^k_2\vadjust{\goodbreak}
\end{equation}
for $k$ large enough, where $D^m_1\subset D^{m+1}_1$ and $D^m_2\supset
D^{m+1}_2; m\ge1$. From (\ref{approximationbm}) it follows that
%
\begin{equation}
\label{e4} \max\bigl\{\mathbb{P}\bigl[D - D^k_1\bigr];
\mathbb{P}\bigl[D^k_2 - D\bigr]\bigr\}\rightarrow0
\qquad\mbox{as } k\rightarrow\infty.
\end{equation}
In fact, by a standard monotone class argument one can easily show that
any set in $\mathcal{F}_t$ satisfies the
above property. Now recall that $\mathcal{O} = \sigma(\mathcal{C})$ where
\[
\mathcal{C} = \bigl\{ E \times\{0\}\dvtx E \in\mathcal{F}_0 \bigr\}
\cup \bigl\{ [s,t) \times E\dvtx s< t; s,t\in \mathbb{Q}_+\cap[0,T], E \in
\mathcal{F}_s \bigr\}.
\]
From (\ref{e3}) and (\ref{e4}) it follows that for each $\Lambda\in
\mathcal{C}$, there exist sequences $O^k_i$
$i=1,2$ so that $O^k_1\subset\Lambda\subset O^k_2$ with $k$ large
enough and
\[
\max\bigl\{\mathbb{P} \bigl[\pi(\Lambda) - \pi\bigl(O^k\bigr)\bigr];\mathbb{P}
\bigl[\pi\bigl(O^k\bigr) - \pi (\Lambda) \bigr]\bigr\}\rightarrow0
\qquad\mbox{as } k\rightarrow\infty.
\]

In order to recover any optional set in $\mathcal{O}$, we shall apply a
routine argument based on the section theorem (see, e.g.,~\cite{yan},
Theorem 4.5) so we omit the details. The proof of the lemma is complete.
\end{pf}


In the remainder of this paper, we will adopt the following
terminology.
%
\begin{definition}
We say that a real-valued process $X$ is a \textit{Wiener functional}
if it is optional w.r.t. the Brownian filtration $\mathbb{F}$ and
$\mathbb{E}|X_{T^k_n}| < \infty$ for every $k,n\ge1$.
\end{definition}


We now embed a given Wiener functional $X$ into a sequence of
$\mathbb{F}^k$ quasi-left continuous bounded variation processes
as
%
\begin{equation}
\label{projection} \delta^k X_t:= X_0 + \sum
_{n=1}^\infty\mathbb{E}\bigl[X_{T^k_n}|
\mathcal {G}^k_n\bigr] \one_{\{ T_{n}^{k} \leq t <
T^k_{n+1} \}},\qquad 0\le t \le T.
\end{equation}

\begin{remark}\label{ucpconv}
The convergence $\delta^kX \rightarrow X$ is just a matter of path
regularity. In fact, as a consequence of~\cite{coquet1},
Theorem 1, we know that if a given Wiener functional $X$ has
continuous paths, then
\[
\mathbb{E}\bigl[X_\cdot|\mathcal{F}^k_\cdot\bigr]
\rightarrow X_{\cdot},\qquad \delta^kX_\cdot\rightarrow
X_\cdot
\]
uniformly in probability as $k\rightarrow\infty$.
\end{remark}
%
\begin{remark}\label{martingaleprojectionremark}
The usual optional stopping theorem implies that any
$\mathbb{F}$-martingale $M$ with $M_0=0$ a.s. admits the
representation
%
\begin{equation}
\label{martingaleprojection} \delta^kM_t =\mathbb{E}
\bigl[M_T|\mathcal{F}^k_t\bigr],\qquad 0\le t \le T.
\end{equation}
In particular, $A^k = \delta^kB$.


\end{remark}

Next, our goal is to establish an explicit decomposition for
the embedded semimartingale skeleton $(\delta^kX)_{k\ge1}$ in terms of
a discrete-type derivative.

\subsection{The approximate decomposition}\label{sec2.1}
In this section, we obtain an explicit Doob--Meyer decomposition for
$\delta^kX$. At first, one should notice that $\{\delta^k X_t\dvtx 0 \le t
\le T\}$ is an
$\mathbb{F}^k$-adapted process with locally integrable variation\vadjust{\goodbreak}
for each $k\ge1$. Moreover, there exists a unique
$\mathbb{F}^k$-predictable process $N^{k,X}$ with locally
integrable variation such that
%
\begin{equation}
\label{absrep} \delta^k X_t - X_0 -
N^{k,X}_t =:M^{k,X}_t,\qquad 0\le t\le T,
\end{equation}
is an $\mathbb{F}^k$-local martingale. The process
$N^{k,X}$ is the $\mathbb{F}^k$-dual predictable projection of $\delta^kX_t- X_0$
which can be taken with continuous paths because $\mathbb{F}^k$ is
quasi left-continuous.

Next we aim at characterizing the elements of the decomposition (\ref
{absrep}). One should notice that since $\mathbb{F}^k$ is not a
completely continuous filtration [see~(\ref{sigmakn})], then $A^k$
cannot have a strong predictable representation.
%
\begin{remark}
Since $A^k$ is a quasi left-continuous martingale and a step process,
then it has the so-called optional representation (see, e.g.,
\cite{yan}, Theorem~13.19 and Example 13.9). That is, every
$\mathbb{F}^k$-local martingale starting from zero is represented by an
optional integral w.r.t. $A^k$.
\end{remark}

In the remainder of this paper, we make use of the optional stochastic
integration w.r.t. $A^k$. We refer the reader to
\cite{Dellacherie1,yan} for all details about optional integrals used
in this paper. We just want to mention here that since the filtration
$\mathbb{F}^k$ is quasi left-continuous, then the related optional
integrals admit the usual operational properties of stochastic
integrals with predictable integrands\vspace*{1pt} (see, e.g.,
\cite{Dellacherie1}, Remark 35, page 346). In this work, we denote by
$\oint_0^t Y_s\,dA^k_s$ the optional integral of an
$\mathbb{F}^k$-optional process $Y$.

We now introduce a process which will play a key role in this work. If
$\delta^kX$ is the
$\mathbb{F}^k$-projection of a Wiener functional $X$, then we define
the following $\mathbb{F}^k$-optional
process
%
\begin{equation}
\label{stochasticderivative} \mathcal{D}\delta^kX:=\sum
_{n=1}^{\infty}\frac{\delta^k X_{T^k_n} -
\delta^kX_{T^k_{n-1}}}{ B_{T^k_n} -
B_{T^k_{n-1}}}\one_{\lbra  T^k_n, T^k_n\rbra }.
\end{equation}
If
%
\begin{equation}
\label{nec} \mathbb{E}\sum_{n=1}^m\bigl|\Delta
\delta^kX_{T^k_n}\bigr|^2 < \infty \qquad\forall m,k
\ge1,
\end{equation}
then
\[
\biggl[\int_0^\cdot\mathcal{D}^2_{s}
\delta^kX\,d\bigl[A^k,A^k\bigr]_s
\biggr]^{1/2} = \Biggl[\sum_{n=1}^{\infty}
\bigl(\delta^kX_{T^k_n} - \delta^kX_{T^k_{n-1}}
\bigr)^2 \one_{\{T^k_n
\leq\cdot\}} \Biggr]^{1/2}
\]
is a locally integrable increasing process for every
$k\ge1$. In this case, there exists a unique $\mathbb{F}^k$-local
martingale $M$ such that for every bounded $\mathbb{F}^k$-martingale
$V$, the process $[M,V] -
\int_0^{\cdot}\mathcal{D}\delta^kX\,d[V,A^k]$ is an
$\mathbb{F}^k$-local martingale and
\[
M_t = \int_0^t
\mathcal{D}_{s}\delta^kX \,dA^k_s -
\biggl[\int_0^{\cdot}\mathcal{D}_{s}
\delta^kX\,dA^k_s \biggr]^{p,k}_t
= \oint_0^t\mathcal{D}_{s}
\delta^kX\,dA^k_s,
\]
where\vspace*{1pt} $\int_0^t\mathcal{D}_{s}\delta^kX \,dA^k_s$ is
interpreted in the Lebesgue--Stieltjes sense. By observing that
$\sum_{0 \le s\le t}\Delta \delta^kX_s=\sum_{0 \le s\le t}\Delta
M^{k,X}_s$ and\vspace*{1pt} the fact that $\delta^kX$ is quasi
left-continuous, we actually have the following optional representation
for the martingale part in the decomposition (\ref{absrep}):
\[
M^{k,X}_t = \oint_0^t
\mathcal{D}_{s}\delta^kX\,dA^k_s;\qquad 0
\le t \le T.
\]
Of course, $\mathcal{D}\delta^kX$ is the unique $\mathbb{F}^k$-optional
process which represents the
martingale $M^{k,X}$ as an optional stochastic integral with respect to
the martingale $A^k$. Let us characterize the remainder term in the
decomposition (\ref{absrep}).
%
\begin{lemma}\label{explicit}
The $\mathbb{F}^k$-dual predictable projection of $\delta^kX - X_0$ is
given by the continuous process
\[
\int_0^tU^{k,X}_sd\bigl
\langle A^k, A^k \bigr\rangle_s,\qquad 0\le t \le
T,
\]
where
$U^{k,X}:=\mathbb{E}_{[A^k]} [\mathcal{D}\delta^kX/\Delta
A^k | \mathcal{P}^k  ]$. Here
$\mathbb{E}_{[A^k]} [\cdot| \mathcal{P}^k  ]$ denotes the
conditional expectation w.r.t. $\mathcal{P}^k$ under the
Dol\'eans measure generated by $[A^k,A^k]$. Moreover,
%
\begin{equation}
\label{uk} U^{k,X}_t= 0 \one_{\{T^k_0=t\}} +
\frac{1}{2^{-2k}}\sum_{n=1}^\infty \mathbb{E}
\bigl[X_t - X_{T^k_{n-1}}| \mathcal{G}^k_{n-1};
T^k_n=t\bigr] \one_{\{ T_{n-1}^k <t \le T_{n}^k\}}.\hspace*{-35pt}
\end{equation}
\end{lemma}
\begin{pf}
The fact that $U^{k,X}=\mathbb{E}_{[A^k]} [\mathcal{D}\delta^kX/\Delta
A^k | \mathcal{P}^k  ]$ is obvious. Let us now characterize
$U^{k,X}$. For this, let us consider the sequence of sigma-algebras
$\mathcal{G}^k_{n-}:= \mathcal{G}^k_{n-1}\vee\sigma(T^k_n)$, $n\ge1$.
We recall that for every $C \in\mathcal{G}^k_{n-}$, there exists a
predictable process $H$ such that $H_{T_{n}^k} =
\one_{C}$ and it is null outside the stochastic interval $\rbra T_{n-1}^k,
T_{n}^k\rbra $ (see~\cite{Bremaud}, Theorem 31,
page 307). Then,
\[
\mathbb{E} \bigl[ \one_{C} \Delta\delta^kX_{T_{n}^k}
\one_{\{
T_{n}^k \leq T \} } \bigr] = \mathbb{E} \bigl[ \one_{C}
U^{k,X}_{T_{n}^k} 2^{-2k} \one_{\{ T_{n}^k \leq T \} } \bigr].
\]
Since $C$ is arbitrary and $U^{k,X}$ a predictable process, it
follows that
\[
\mathbb{E} \bigl[ \Delta\delta^kX_{T_{n}^k}
\one_{\{ T_{n}^k
\leq T \} } \mid\mathcal{G}^k_{n-} \bigr] =
U^{k,X}_{T_{n}^k} 2^{-2k} \one_{\{ T_{n}^k \leq T \} }.
\]
Then, one version of the conditional expectation can be written as (\ref
{uk}). The proof of the lemma is complete.
\end{pf}

The next result describes an explicit expression for the predictable
bracket of $A^k$ in terms of the density $f^k$ and the distribution
function $F^k$ of $T^k_1$ (see, e.g.,~\cite{burq} for the corresponding
formulas).
%
\begin{lemma}\label{anglebracket}
The predictable bracket of $A^k$ is an absolutely continuous process
where the Radon--Nikodym derivative process is given by
%
\begin{equation}
\label{intensity} h^k_t = 2^{-2k}\sum
_{n=1}^{\infty} \lambda^k_{t-T_{n-1}^k}
\one_{ \{
T_{n-1}^k < t \leq T_{n}^k\}},\qquad 0\le t\le T,
\end{equation}
where $\lambda^k_t = \frac{f^k_t}{1-F^k_t}$, $0\le t\le T$.
\end{lemma}
\begin{pf}
Since $A^k$ is a quasi left-continuous point process where the
difference of the jumping times $\{(T_{n}^k - T_{n-1}^k); n\ge1\}$ is
a sequence of i.i.d. absolutely continuous random variables, then it is
well known that $\langle A^k,A^k\rangle$ has absolutely continuous
paths. A straightforward but lengthy calculation together with
\cite{lipster}, Theorem 18.2, yields (\ref{intensity}).
\end{pf}

Summing up all previous results of this section, we then arrive at the
following representation.
%
\begin{proposition}\label{appdecomp}
If $X$ is a Wiener functional satisfying assumption (\ref{nec}), then
the $\mathbb{F}^k$-special semimartingale decomposition
$(M^{k,X},N^{k,X})$ in (\ref{absrep}) is actually given by
%
\begin{equation}
\label{explicitdec} \delta^kX_t = X_0 +
\oint_0^t\mathcal{D}_s\delta^kX
\,dA^k_s + \int_0^t
U^{k,X}_sh^k_s\,ds,\qquad 0\le t\le T.\hspace*{-20pt}
\end{equation}
\end{proposition}

\section{Weak decomposition of Wiener functionals}\label{sec3}\label{convergencemain}
In this section we are interested in providing readable
conditions on a given Wiener functional $X$ in such way that
\[
X=\lim_{k\rightarrow\infty}\delta^kX;\qquad
M^X=\lim_{k\rightarrow\infty}M^{k,X};\qquad
N^X=\lim_{k\rightarrow\infty}N^{k,X}
\]
in a suitable topology. Under such assumptions, we are able
to decompose $X$ into a unique orthogonal decomposition which is
similar in nature to weak Dirichlet processes (see, e.g.,
\cite{coviello} and other references therein)
\[
X_t = X_0 + M^X_t +
N^X_t,
\]
where $M^X$ is a martingale and $N^X$ is an adapted
process whose specific type of \textit{covariation} (see
Definition~\ref{covariation}) w.r.t. Brownian motion is
null.
\subsection{Weak convergence and primary decomposition}\label{sec3.1}\label{converg}

In this section we investigate the convergence of our preliminary
decomposition (\ref{explicitdec}) given in
terms of the approximation scheme $(A^k,\mathbb{F}^k)$. By carefully
choosing a suitable topology on the space
of processes, our strategy will be fully based on the information given
by the quadratic variation of the
martingale component in (\ref{explicitdec}).

Let $\mathrm{B}^p (\mathbb{F})$ be the set of all $\mathbb{F}$-optional
processes and which
are $1\le p< \infty$ B\"{o}chner integrable in the sense that
%
\begin{equation}
\label{bochner} \|X\|^p_{\mathrm{B}^p} = \mathbb{E}\bigl|X^*_T\bigr|^p
< \infty,
\end{equation}
where $X^*_T:={\sup_{0\le t \le T}}|X_t|$. Of course,
$\mathrm{B}^p(\mathbb{F})$ endowed with the norm \mbox{$\|\cdot\|_{\mathrm
{B}^p}$} is a Banach space, where the
subspace $\mathrm{H}^p(\mathbb{F})$ of the $\mathbb{F}$-martingales
starting from zero is closed. Recall that the topological dual $\mathrm
{M}^q(\mathbb{F})$ of $\mathrm{B}^p(\mathbb{F})$ is the space of
processes $A = (A^{\mathit{pr}},
A^{\mathit{pd}})$ such that:

\begin{longlist}[(ii)]
\item[(i)]
$A^{\mathit{pr}}$ and $A^{\mathit{pd}}$ are right-continuous of bounded variation
such that $A^{\mathit{pr}}$ is $\mathbb{F}$-predictable with $A^{\mathit{pr}}_0=0$ and
$A^{\mathit{pd}}$ is $\mathbb{F}$-optional and purely discontinuous.

\item[(ii)] $\operatorname{Var}(A^{\mathit{pd}}) +
\operatorname{Var}(A^{\mathit{pr}})\in L^q; \frac{1}{p}+\frac{1}{q}=1$,
\end{longlist}
where $\operatorname{Var}(\cdot)$ denotes the total variation of a bounded variation
process on the interval $[0,T]$. The space $\mathrm{M}^{q}(\mathbb{F})$
has the strong topology given by
\[
\|A\|_{M^q}: = \bigl\|{\operatorname{Var}}\bigl(A^{\mathit{pr}}\bigr)\bigr\|_{L^q}
+ \bigl\|{\operatorname{Var}}\bigl(A^{\mathit{pd}}\bigr)\bigr\|_{L^q}.
\]
The duality pair is given by
\[
(A,X):= \mathbb{E}\int_{0}^T X_{s-}
\,dA^{\mathit{pr}}_s + \mathbb{E}\int_0^T
X_s\,dA^{\mathit{pd}}_s;\qquad X \in\mathrm{B}^p(
\mathbb{F}),
\]
where the following estimate holds:
\[
\bigl|(A, X)\bigr| \le\|A\|_{M^q} \|X\|_{\mathrm{B}^p}
\]
for every $A \in\mathrm{M}^q(\mathbb{F})$, $X \in\mathrm
{B}^p(\mathbb{F})$ such that $1\le p < \infty$ and $\frac{1}{p} + \frac
{1}{q}=1$. We denote $\sigma(\mathrm{B}^p,\mathrm{M}^q)$ the weak
topology of $\mathrm{B}^p(\mathbb{F})$.

In this work, the indexes $p=1,2$ will play a key role in our
convergence results, in particular, the subspaces $\mathrm{H}^p$ for
$p=1,2$. See the works~\cite{Dellacherie1,Dellacherie3,meyer1} for detailed discussions on the weak topology of $\mathrm
{B}^p(\mathbb{F})$ restricted to the subspace of martingales $\mathrm
{H}^p(\mathbb{F})$.

In this article it will be also useful to work with the following
notion of convergence. Actually, one can show that the set $\Lambda^{\infty}$ of the $\mathbb{F}$-optional bounded variation processes of
the form
\[
C = g\one_{\{S \leq\cdot\}};\qquad g \in L^{\infty}(\mathcal{F}_S),
S\mbox{ is an } \mathbb{F}\mbox{-stopping time (bounded by T)},
\]
fulfills the Banach space $\mathrm{B}^1(\mathbb{F})$ in the
sense that
%
\begin{equation}
\label{weakweak} \|X\|_{\mathrm{B}^1}=\sup \bigl\{ \bigl|(X, C)\bigr|; C\in
\Lambda^{\infty}, \|C\|_{\mathrm{M}^{\infty}}\le1 \bigr\}.
\end{equation}

Relation (\ref{weakweak}) is given in~\cite{Dellacherie3}, Lemma 1, and
therefore we may also endow
$\mathrm{B}^1(\mathbb{F})$ with the $\sigma(\mathrm{B}^1, \Lambda^{\infty})$-topology induced by the family of
seminorms
\[
X \mapsto\bigl|(X, C)\bigr|;\qquad C\in\Lambda^{\infty}.
\]
%
\begin{remark}
Obviously, $\sigma(\mathrm{B}^1, \Lambda^{\infty})$ is weaker then
$\sigma(\mathrm{B}^1,\mathrm{M}^\infty)$. However,
relation (\ref{weakweak}) says that $\Lambda^{\infty}$ is a norming
subset of $\mathrm{M}^{\infty}$ and therefore
$\Lambda^{\infty}$ is $w^*$-dense in $\mathrm{M}^\infty$.
\end{remark}
%
\begin{remark}\label{moko}
A result due to Mokobodzki~\cite{meyer1} states that if $X^n$
is a sequence of optional processes such that $\sup_{0\le t \le
T}|X^n_t|$ is uniformly integrable and for every $S$ stopping time the
sequence $X^n_S$ converges weakly in $L^1$ relatively to $\mathcal
{F}_S$, then there exists an optional process $X$ such that
$X^n\rightarrow X$ in $\sigma(\mathrm{B}^1,\mathrm{M}^\infty)$. As a
consequence, if $X^n\rightarrow X$ in $\sigma(\mathrm{B}^1,\Lambda^{\infty})$ and $\sup_{0\le t \le T}|X^n_t|$ is uniformly integrable,
we do have convergence in $\sigma(\mathrm{B}^1,\mathrm{M}^\infty)$
(see also Dellacherie, Meyer and Yor~\cite{Dellacherie3} for
more details).
\end{remark}

In the remainder of this paper, we shall write $\mathrm{B}^p$
($H^p$) to denote the space of B\"{o}chner integrable process
($p$-integrable martingales starting from zero) satisfying (\ref{bochner})
endowed with the Brownian filtration $\mathbb{F}$. We now introduce the
following quantity which will play a crucial role in this work.
%
\begin{definition}\label{energy}
We say that a given Wiener functional $X$ has \textit{finite
energy} along the filtration family $(\mathbb{F}^k)_{k\ge1}$ if
%
\begin{equation}
\label{p-variation} \mathcal{E}_2(X):=\sup_{k\ge1}\mathbb{E}
\sum_{n=1}^{\infty} \bigl|\Delta\delta^k
X_{T^k_n}\bigr|^2\one_{\{ T_{n}^{k} \leq T \}}< \infty.
\end{equation}
\end{definition}
%
\begin{remark}
The above definition is similar in spirit to the classical notion
of energy (e.g.,~\cite{gr,coquet2}), but with one fundamental
difference: The relevant information contained in the energy of
$X$ comes only from the sigma-algebras $\mathcal{G}^k_n$ which
reveal the information generated by the jumps of the projected
Brownian motion $A^k$ up to the stopping time $T^k_n$. Moreover,
$\mathcal{E}_2(X)=\sup_{k\ge1}\mathbb{E}[M^{k,X},M^{k,X}]_T $.
\end{remark}

It is natural to ask what happens without conditioning on the
information flow $\{\mathcal{G}^k_n;k,n\ge1\}$. The following lemma
answers this question.
%
\begin{lemma}\label{energyresult}
If $X$ is a Wiener functional, then
%
\begin{equation}
\label{energyestimate} \mathcal{E}_2(X)\le\sup_{k\ge1}
\mathbb{E} \sum_{n\ge1} (X_{T^k_n} -
X_{T^k_{n-1}})^2\one_{ \{ T^k_n\le T \}}.
\end{equation}
\end{lemma}
\begin{pf}
It is sufficient to check that $\mathbb{E}[\mathcal{H}^{k,X}_t|\mathcal
{F}^k_t] =\mathcal{D}_t\delta^kX$ on $\{T^k_n\le t < T^k_{n+1} \}$ for
each $k,n\ge1$ where
\[
\mathcal{H}^{k,X}:=
\sum_{n=1}^{\infty}
\frac
{X_{T^k_n}-X_{T^k_{n-1}}}{B_{T^k_n}-B_{T^k_{n-1}}}
{\bigl[\!\bigl[  T^k_n, T^k_{n+1}\bigr[\!\bigr[}.
\]
But this is a straightforward consequence of the strong Markov property
of the Brownian motion. A simple application of Jensen inequality and
the $\mathbb{F}^k$-optional duality establishes (\ref{energyestimate}).
\end{pf}

We are now in position to study convergence of the decomposition
given in (\ref{explicitdec}). In the sequel, we fix an element
$X\in\mathrm{B}^1$ and let $(M^{k,X},N^{k,X})$ be the associated
canonical decomposition expressed in (\ref{explicitdec}). In order to
find a candidate for the limit
of $M^{k,X}$, let us introduce the following family of $\mathbb{F}$-martingales:
%
\begin{equation}
\label{fundamentalmartingale} Z^{k,X}_t:=\mathbb{E}
\bigl[M^{k,X}_T|\mathcal{F}_t\bigr];\qquad 0\le t \le
T; k\ge1.
\end{equation}

In order to prove convergence of $M^{k,X}$ to an $\mathbb
{F}$-martingale we may use some standard compactness
arguments.
%
\begin{lemma} \label{tcb4} The sequence of random
variables \mbox{$\{ [M^{k,X},M^{k,X}]_T^{1/2}\dvtx k\ge1 \}$} is uniformly
integrable if, and only if, the sequence of stochastic process
$\{Z^{k,X}\dvtx k\ge1 \}$ is weakly relatively compact in
$\mathrm{H}^1$.
\end{lemma}
\begin{pf}
The proof is a routine argument based on the Doob and Burk\-holder
inequalities together with~\cite{Dellacherie3}, Theorem 1, so we omit
the details.~%
\end{pf}
%
\begin{remark}\label{h2compact}
By the Doob maximal inequality one should notice that if $\mathcal
{E}_2(X)<\infty$, then $\{Z^{k,X};k\ge1 \}$ is a bounded sequence in
$\mathrm{H}^2$ which also implies that it is an $\mathrm{H}^2$-weakly
sequentially compact set.
\end{remark}
%
\begin{lemma}\label{c2}
If $S$ is an $\mathbb{F}$-stopping time, then there exists a
sequence of positive random variables $(S_ k)_{k\ge1}$ such that
$S_k$ is an $\mathbb{F}^k$-stopping time for each $k\ge1$ and
$\lim_{k\rightarrow\infty} \mathbb{P}(S_k = S)=1$. Moreover, for
any $G\in\mathcal{F}_{S}$ there exists a sequence of sets
$(G^k)_{k\ge1}$ such that $G^k\in\mathcal{F}^k_{S_k}$, $G^k
\subset G \cap\{S < \infty\}$ for every $k\ge1$, and
\[
\lim_{k\rightarrow\infty}\mathbb{P}\bigl[G\cap\{S < \infty\} - G^k\bigr]
= 0.
\]
\end{lemma}
\begin{pf}
Let $S$ be an arbitrary $\mathbb{F}$-stopping time. Since the graph
$\lbra S\rbra $ belongs to $\mathcal{O}$, we may find a sequence $(O^k)_{k\ge
1}$ satisfying item (iii) in Lemma~\ref{smoothfiltration}.
For an arbitrary $\varepsilon>0$, let $k$ be large enough in such way that
\[
\mathbb{P}\bigl[\pi\bigl(\lbra S\rbra\bigr) - \pi\bigl(O^k\bigr)
\bigr] < \varepsilon/ 2.
\]
From the standard section theorem there exists an $\mathbb
{F}^k$-stopping time $S_k$ such that
\[
\lbra S_k\rbra\subset O^k \subset \lbra S\rbra
\quad\mbox{and}\quad \mathbb{P}\bigl[\pi\bigl(O^k\bigr)\bigr]\le
\mathbb{P}[S_k < \infty] + \varepsilon/ 2.
\]
Then it follows that $\mathbb{P}[\pi(S) - \pi(\lbra S_k\rbra )] = \mathbb
{P}[S_k \neq S] < \varepsilon$ for $k$ large
enough. This allows us to conclude the first part of the lemma.\vadjust{\goodbreak} For the
second part, let us recall that $\mathcal{F}_{S} \cap\{ S < \infty\}
= \{\Phi^{-1} (O)\dvtx O \in
\mathcal{O}\}$, where $\Phi(w) = (S(w),w)$ for any $w \in\{S < \infty
\}$. Then, for any $G \in\mathcal{F}_S$,
there exists an optional set $J \in\mathcal{O}$ such that $G \cap\{ S
< \infty\} = \Phi^{-1} (J)$. We denote
by $J^k$ the sequence of sets satisfying item (iii) in Lemma \ref
{smoothfiltration} and $G^{k} = \Phi^{-1}_{k}
(J^k)$, where $\Phi_{k}(w) = (S^k(w),w)$ for any $w \in\{S^k < \infty
\}$ and $(S_k)_{k\ge1}$ is the sequence
of stopping times obtained from the first part. Then we conclude that
$\mathbb{P}[G \cap\{S<\infty\} - G^k]\le
\mathbb{P}[\pi(J) - \pi(J^k)] \rightarrow0$ as $k\rightarrow\infty$.
\end{pf}

Summing up the above lemmas we arrive at the following result.
%
\begin{proposition}\label{relcompact}
Assume that a Wiener functional $X$ satisfies $\mathcal{E}_2(X)<\infty
$. Then the set $\{M^{k,X}_{\cdot}; k\ge
1\}$ is $\sigma(B^2,\mathrm{M}^{2})$---relatively sequentially
compact where every limit point belongs to $\mathrm{H}^2$.
\end{proposition}
\begin{pf}
From Remark~\ref{h2compact}, we know that $\mathcal{Z}=\{Z^{k,X}; k\ge
1\}$ is weakly relatively compact in
$\mathrm{H}^2$ and therefore any sequence in $\mathcal{Z}$ admits a
weakly convergent subsequence in $\mathrm{H}^2$.
With a slight abuse of notation, let us denote by $Z^{k,X}$ this
convergent subsequence in $\mathrm{H}^2$ and $Z$ the respective $\mathrm
{H}^2$-martingale limit point. Let us fix an $\mathbb{F}$-stopping time
$S$ which is bounded by the terminal time $T$. We claim that
$M^{k,X}\rightarrow Z$ in $\sigma(\mathrm{B}^2, M^{2})$. For this, at
first we show that the convergence holds in the $\sigma(\mathrm
{B}^1,\Lambda^{\infty})$-topology. In other words,
\[
\lim_{k\rightarrow\infty}\int M^{k,X}_Sg\,d\mathbb{P}=\int
Z_Sg\,d\mathbb{P}
\]
%
holds for every $g\in L^{\infty}(\mathcal{F}_S)$. Recall that
it is sufficient to prove for indicator functions $g= \one_{G}$ where
$G\in\mathcal{F}_S$. From Lemma~\ref{c2} there exists a sequence of
stopping times $S_{k}$ and $G^{k} \in\mathcal{F}^{k}_{S^{k}}$
satisfying $\lim_{k\rightarrow\infty} \mathbb{P} [ G - G^{k} ]=0$. By
construction, one should notice from the proof of Lemma~\ref{c2} that
$S_k = S$ on $G^k$ for every $k\ge1$. Moreover, the martingale
property yields
\begin{eqnarray*}
\int_{G} Z^{k,X}_S \,d\mathbb{P} &=&
\int_{G-G^{k}} M^{k,X}_T \,d\mathbb{P} + \int
_{G^k} \mathbb{E} \bigl[ M^{k,X}_T \mid
\mathcal{F}^{k}_{S^{k}} \bigr] \,d\mathbb{P}
\\
&=& \int_{G-G^{k}} M^{k,X}_T \,d\mathbb{P} -
\int_{G - G^{k}} M^{k,X}_S \,d\mathbb{P} + \int
_{G} M^{k,X}_S \,d\mathbb{P}.
\end{eqnarray*}
Therefore, the uniform integrability assumption yields
\begin{eqnarray*}
&&
\biggl|\int_{G}Z^{k,X}_S\,d\mathbb{P} - \int
_{G}M^{k,X}_S\,d\mathbb{P} \biggr| \\
&&\qquad= \biggl|\int
_{G-G^k}M^{k,X}_T\,d\mathbb{P} - \int
_{G-G^k}M^{k,X}_S\,d\mathbb{P} \biggr|
\longrightarrow0 \qquad\mbox{as } k\rightarrow\infty.
\end{eqnarray*}

This shows that $\lim_{k\rightarrow\infty}( M^{k,X}, C) = (Z, C )$ for
every $C\in \Lambda^{\infty}$ and therefore we may conclude that
$\lim_{k\rightarrow\infty}M^{k,X}=Z$ in the\vadjust{\goodbreak}
$\sigma(\mathrm{B}^1,\Lambda^{\infty})$-topology. The uniform
integrability of $\{\sup_{0\le t \le T}|M^{k,X}_t|; k\ge1\}$ and Remark
\ref{moko} allow us to conclude that $M^{k,X}\rightarrow Z$ weakly in
$\mathrm {B}^1$. We now claim that the subsequence $M^{k,X}$ actually
converges to $Z$ weakly in $\mathrm{B}^2$. For this, let us consider an
arbitrary linear functional $Y\in\mathrm{M}^2$ given by $Y =
(Y^{\mathit{pr}}, Y^{\mathit{pd}})$. Let $(S_n)_{n\ge1}$ be a
localizing\vspace*{1pt} sequence of $\mathbb {F}$-stopping times such that
$\operatorname{Var} (Y^{\mathit{pr}}_{\wedge S_n})$ and
$\operatorname{Var} (Y^{\mathit{pd}}_{\wedge S_n})$ are bounded for
every $n\ge1$. Let us denote by $Y^n$ the respective stopped linear
functional $Y^n\in\mathrm{M}^\infty $; $n\ge1$. The finite energy
assumption yields
%
\begin{eqnarray}
\label{comh2} \bigl|\bigl(Y,M^{k,X}\bigr)-(Y,Z)\bigr|&\le&\bigl|\bigl(Y^n,M^{k,X}
\bigr)-\bigl(Y^n,Z\bigr)\bigr|\nonumber\\[-8pt]\\[-8pt]
&&{} +\bigl\|Y^n-Y\bigr\|_{M^2}
\bigl(\| Z\|_{B^2} +\mathcal{E}^{1/2}_2(X) \bigr).\nonumber
\end{eqnarray}
Since $\|Y^n-Y\|_{M^2}\rightarrow0$ as $n\rightarrow\infty$ and $Y^n
\in\mathrm{M}^\infty$ for every $n\ge1$, we shall use the $\mathrm
{B}^1$-weak convergence of $M^{k,X}$ to $Z$ and (\ref{comh2}) to
conclude that $\lim_{k\rightarrow\infty}M^{k,X}=Z$ weakly in $\mathrm
{B}^2$. In other words, $\{M^{k,X};k\ge1 \}$ is a $\mathrm
{B}^2$-weakly, relatively, sequentially compact set where every limit
point belongs to $\mathrm{H}^2$.
\end{pf}

In the sequel, we introduce a covariation notion which plays a key role
in the numerical scheme of the stochastic derivative. We stress here
that it is not our purpose to give a more general definition of a
quadratic variation. Instead, we only need a slightly different type of
approximation due to the (a priori) lack of regularity of the Wiener
functionals.
%
\begin{definition}\label{covariation}
Let $X$ and $Y$ be Wiener functionals with $\mathbb{F}^k$-projections,
$\delta^kX$ and
$\delta^kY$, respectively. We say that $X$ admits the
$\delta$-\textit{covariation} w.r.t. $Y$ if the limit
%
\begin{equation}
\label{deltacov} \langle X,Y\rangle^{\delta}_t:=
\lim_{k\rightarrow\infty}\bigl[\delta^k X, \delta^k Y
\bigr]_t
\end{equation}
exists weakly in $L^1$ for every $t\in[0,T]$.
\end{definition}
%
\begin{remark}
In the particular case of Brownian semimartingales, one can easily
check that the $\delta$-covariation coincides with the usual quadratic
variation by using Lemma~\ref{energyresult} and Proposition
\ref{keyresult}.
\end{remark}
%
\begin{remark}
The reason for choosing the $L^1$-weak topology for the covariation is
due to the lack of path regularity of processes which represents
Brownian martingales. We will see that the $L^1$-weak topology is the
correct one if one attempts to get a robust approximation scheme in
full generality without requiring additional assumptions (see Remark
\ref{mainint3}).
\end{remark}

Next, we prove some technical results which will allow us to state
Theorem~\ref{mainresult} which is the main result of this section. Not
surprisingly, the quadratic variation and energy
notions will play a key role in our result.\vadjust{\goodbreak}

\begin{lemma}\label{convinteg}
Let $H_\cdot=\mathbb{E}  [\one_{G}| \mathcal{F}_\cdot ]$ and
$H^k_\cdot= \mathbb{E} [ \one_{G} |
\mathcal{F}_{\cdot}^k ]$ be positive and uniformly integrable
martingales with respect to the filtrations
$\mathbb{F}$ and $\mathbb{F}^k$, respectively, where $G \in\mathcal
{F}_T$. If $W\in\mathrm{H}^\alpha(\mathbb{F})$ for some $\alpha>2$ then
\[
\biggl\|\int_{0}^{\cdot} H_s
\,dW_s - \oint_{0}^{\cdot} H^k_s
\,d\delta^kW_s \biggr\|_{\mathrm{B}^2} \rightarrow0 \qquad\mbox{as }
k\rightarrow\infty.
\]
\end{lemma}
\begin{pf}
Throughout the proof we write $C$ to denote a positive constant which
may differ from line to line. Let us write
\begin{eqnarray*}
&&
\int_{0}^t H_s \,dW_s
- \oint_{0}^t H^k_s d
\delta^k W_s \\
&&\qquad= \int_{0}^t
\bigl[ H_s- H^k_s\bigr] \,dW_s +
\biggl[ \int_{0}^t H^k_s
\,dW_s - \oint_{0}^t H^k_s
\,d\delta^k W_s \biggr]
\\
&&\qquad=:T^1_k(t) + T^2_k(t), \qquad 0\le t
\le T.
\end{eqnarray*}
From the weak convergence of $\mathbb{F}^k$ to
$\mathbb{F}$ [see (ii) in Lemma~\ref{smoothfiltration}] and the
fact that $H$ is a continuous process it follows that
%
\begin{equation}
\label{a2} H^k\rightarrow H\qquad
\mbox{uniformly in probability as } k
\rightarrow\infty.
\end{equation}

Burkholder and H\"older inequalities yield
\[
\biggl\| \int_0^\cdot\bigl(H^k_t-H_t
\bigr)\,dW_t \biggr\|^2_{B^2} \le C\mathbb{E}^{{1}/{p}}
\sup_{0\le t \le T}\bigl|H_t - H^k_t\bigr|^{2p}
\mathbb{E}^{{1}/{q}}[W,W]^q_T
\]
for $q=\frac{\alpha}{2}$ and $p=\frac{\alpha}{\alpha-2}$ with
$\alpha>2$. Therefore, we may conclude that $T^1_k\rightarrow0$ in
$\mathrm{B}^2$ as $k\rightarrow\infty$. In order to prove that $T^2_k$
vanishes when $k\rightarrow0$, we split it into the following terms:
%
\begin{eqnarray}
\label{a3} T^2_k(t) &=& \int_{0}^t
\bigl[ H^k_s - H^k_{s-}
\bigr]\,dW_s - \oint_{0}^t \bigl[
H^k_s - H^k_{s-} \bigr]\,d
\delta^kW_s\nonumber\\[-8pt]\\[-8pt]
&&{} + \int_{0}^t
H^k_{s-} \,dW_s - \int_{0}^t
H^k_{s-}\,d\delta^kW_s.\nonumber
\end{eqnarray}
%
%
We shall estimate in the same way
\[
\biggl\|\int_0^{\cdot} \bigl[ H^k_s
- H^k_{s-} \bigr]\,dW_s \biggr\|^2_{\mathrm{B}^2}
\le C\mathbb{E}^{1/p}\sup_{n\ge
1}\bigl|H^k_{T^k_n}-H^k_{T^k_{n-1}}\bigr|^{2p}
\one_{\{T^k_n\le T\}}\mathbb {E}^{{1}/{q}}[W,W]^{q}_T
\]
for $p$ and $q$ as above. One can easily check (see, e.g.,
Lemma~\ref{gkn}) that ${\sup_{n\ge
1}}|H^k_{T^k_n}-H^k_{T^k_{n-1}}|\rightarrow0$ as $k\rightarrow\infty$
in $L^r$ for any $r>1$. Therefore,
\[
\biggl\|\int_0^{\cdot} \bigl[ H^k_s
- H^k_{s-} \bigr]\,dW_s \biggr\|_{\mathrm{B}^2}
\rightarrow0 \qquad\mbox{as } k\rightarrow\infty.
\]
By using the representation
$\delta^kW_t = \oint_0^t\mathcal{D}_s\delta^kW\,dA^k_s=\mathbb
{E}[W_T|\mathcal{F}^k_t]$, we shall use the Doob maximal inequality to
estimate in the same way
\[
\biggl\|\oint_0^{\cdot} \bigl[H^k_s -
H^k_{s-} \bigr]\,d\delta^kW_s
\biggr\|^2_{\mathrm{B}^2}\le C\mathbb {E}^{{1}/{\alpha}}|W|^\alpha_T
\mathbb{E}^{1/p}\sup_{0\le t \le T} \bigl|H^k_t -
H^k_{t-}\bigr|^{2p} \rightarrow0
\]
as $k\rightarrow\infty$. It remains to estimate the last
part in (\ref{a3}). We claim that $(H^k_{-},\delta^k W)$ satisfies the
assumptions of~\cite{protter}, Theorem 2.7. To see this, one notices
that the linearity of the conditional expectation, (\ref{a2}) and the
path continuity of $H$ and $W$ yield
\[
H^k-H + \delta^kW-W\rightarrow0 \qquad\mbox{uniformly in
probability.}
\]
Since $\lim_{k\rightarrow\infty}\mathbb{E}\sup_{0\le t\le
T}|H^k_t-H^k_{t-}|=0$ we actually have $(H^k_{-},\delta^kW)\rightarrow
(H,W)$ in probability on the two-dimensional Skorohod space. Moreover,
a simple application of the maximal Doob and Burkholder inequalities
ensures that $\delta^kW$ satisfies~\cite{protter}, assumption C2.7. Therefore,
\[
\int_0^\cdot H^k_{s-}\,d
\delta^kW_s \rightarrow\int_0^\cdot
H_s\,dW_s \qquad\mbox{uniformly in probability}.
\]
Of course, $\int_0^\cdot H^k_{s-}\,dW_s \rightarrow\int_0^\cdot H_s\,dW_s$ uniformly in probability. By using the assumption
that $W\in\mathrm{H}^\alpha$ for $\alpha>2$, we have $\int_0^\cdot
H^k_{s-}\,d\delta^kW_s + \int_0^\cdot H^k_{s-}\,dW_s$ is bounded in $\mathrm
{B}^\alpha$. This shows that $T^2_k\rightarrow0$ in $\mathrm{B}^2$ as
$k\rightarrow\infty$ and therefore the proof is complete.
\end{pf}


The next result is fundamental for the approach taken in this work
since it allows us to compute the $\delta$-covariation under a
compactness assumption.

\begin{lemma} \label{wdp} Let $X$ be a finite
energy Wiener functional with the $\mathbb{F}^k$-decomposition
given by $(M^{k,X},N^{k,X})$. Let $\{M^{k_i,X};i\ge1\}$ be a $\mathrm
{B}^2$-weakly convergent
subsequence such that $\lim_{i\rightarrow\infty}M^{k_i,X}=Z$, where
$Z\in\mathrm{H}^2$. If $W\in\mathrm{H}^2$, then
%
\begin{equation}
\label{LLL1} \lim_{i\rightarrow\infty}\bigl[M^{k_i,X}, \delta^{k_i}W
\bigr]_{t}= [Z,W]_{t} \qquad\mbox{weakly in } L^1
\end{equation}
for every $t \in[0,T]$.
\end{lemma}
\begin{pf}
With a slight abuse of notation, let $Z^{k,X}$ be the
$\mathbb{F}$-martingale subsequence obtained
from (\ref{fundamentalmartingale}), Remark~\ref{h2compact} and
Proposition~\ref{relcompact} such that $\lim_{k\rightarrow\infty
}Z^{k,X}=Z$ and $\lim_{k\rightarrow\infty}M^{k,X}=Z$ in
$\sigma(\mathrm{B}^2,\mathrm{M}^{2})$. By using representation (\ref
{martingaleprojection}) and the weak convergence $\mathbb
{F}^k\rightarrow\mathbb{F}$, we notice that $\delta^kW\rightarrow W$
in $\sigma(\mathrm{B}^2,\mathrm{M}^2)$ as $k\rightarrow\infty$ for
each $W\in\mathrm{H}^2$. Thanks to~\cite{Dellacherie3}, Theorem 7, we
know that $[Z^{k,X}, U]_t {\rightarrow} [Z,U]_t$ weakly in $L^1(\mathbb
{P})$ for every $t\in[0,T]$ and $U$ a BMO $\mathbb{F}$-martingale.
Given $G \in
\mathcal{F}_T$, let us consider the $\mathbb{F}^k$-martingale
$H^k_\cdot= \mathbb{E} [ \one_{G} | \mathcal{F}_{\cdot}^k ]$ and $W$
a bounded Brownian martingale.
At first, one should notice that the finite energy assumption gives\vadjust{\goodbreak}
$M^{k,X}\in
\mathrm{H}^2(\mathbb{F}^k)$ for every $k\ge1$. By using the $\mathbb{F}^k$-dual
optional projection property we shall write
\[
\mathbb{E} \bigl[ \one_{G} \bigl[M^{k,X},
\delta^kW\bigr]_t \bigr] = \mathbb{E}
\bigl[M^{k,X}, J^{k} \bigr]_t = \mathbb{E} \bigl[
M^{k,X}_{t} J^{k}_{t} \bigr],\qquad 0\le t
\le T,
\]
where $J^{k}$ is the $\mathbb{F}^k$-square integrable martingale
given by the optional integral $\oint H^k \,d\delta^kW$. In the same
manner, we have that
\[
\mathbb{E} \bigl[ \one_{G} \bigl[Z^{k,X}, W
\bigr]_t \bigr] = \mathbb{E} \bigl[Z^{k,X}, J
\bigr]_{t} = \mathbb{E} \bigl[ Z^{k,X}_t
J_t \bigr],\qquad 0\le t\le T,
\]
where $J$ is the stochastic integral $\int H \,dW$ and $H=\mathbb{E}[\one_{G}|\mathcal{F}_{\cdot}]$. Moreover,
\begin{eqnarray*}
\mathbb{E} \bigl[ Z^{k,X}_t J_t \bigr] -
\mathbb{E} \bigl[ M^{k,X}_t J^{k}_t
\bigr] &=& \mathbb{E} \bigl[ M^{k,X}_t \bigl( J_t
- J^{k}_t \bigr) \bigr] - \mathbb{E} \bigl[ \bigl(
M^{k,X}_t - Z^{k,X}_t \bigr)
J_t \bigr]
\\[-1pt]
&=\!&: T^k_1(t) + T^k_2(t),\qquad 0\le
t \le T.
\end{eqnarray*}
We fix $t\in[0,T]$ and we notice that it is sufficient to
prove that $T^k_1(t) + T^k_2(t)\rightarrow0$
as $k\rightarrow\infty$. The first term $\lim_{k\rightarrow\infty
}T^k_1(t)= 0$ because of the finite energy
assumption and Lemma~\ref{convinteg}. By noting that both subsequences
$Z^{k,X}$ and $M^{k,X}$ converge
to the same limit in $\sigma(\mathrm{B}^2, \mathrm{M}^{2})$, we shall
take the linear functional $Y_\cdot=\one_{\{t\le\cdot\}}J_t \in
\mathrm{M}^2$ to conclude that $T^k_2(t)\rightarrow0$ as $k\rightarrow
\infty$. Therefore, (\ref{LLL1}) holds for any bounded martingale $W$.
If $W\in\mathrm{H}^2$ then we shall take a sequence of bounded
martingales $W^n$ such that $W^n\rightarrow W$ in $\mathrm{H}^2$ as
$n\rightarrow\infty$. Moreover, Burkholder and maximal Doob
inequalities yield
%
\begin{eqnarray}
\label{boundedseq}
&&
\bigl|\mathbb{E}\one_{G}\bigl[M^{k,X},
\delta^kW\bigr]_t-\mathbb{E}\one_{G}[Z,W]_t
\bigr|\nonumber\\[-1pt]
&&\qquad
\le C \mathbb{E} \bigl|\bigl[M^{k,X},\delta^k
\bigl(W-W^n\bigr)\bigr]_t \bigr|
\nonumber
\\[-1pt]
&&\qquad\quad{}
+ \bigl|\mathbb{E}\one_{G}\bigl[M^{k,X},\delta^kW^n
\bigr]_t- \mathbb{E}\one_{G}\bigl[Z,W^n
\bigr]_t \bigr|
\nonumber\\[-8pt]\\[-8pt]
&&\qquad\quad{}+C\mathbb{E} \bigl|\bigl[Z,W^n-W\bigr]_t \bigr|
\nonumber
\\[-1pt]
&&\qquad\le C \bigl(\mathcal{E}_2^{1/2}(X)+\|Z\|_{\mathrm{B}^2}
\bigr)\bigl\| W^n_T-W_T\bigr\|_{L^2}
\nonumber
\\[-1pt]
&&\qquad\quad{}+ \bigl|\mathbb{E}\one_{G}\bigl[M^{k,X},\delta^kW^n
\bigr]_t- \mathbb{E}\one_{G}\bigl[Z,W^n
\bigr]_t \bigr|.\nonumber
\end{eqnarray}
Inequality (\ref{boundedseq}) and the previous arguments
allow us to conclude the proof.
\end{pf}

Next, we give a necessary and sufficient condition for the existence of
the martingale limit.
%
\begin{proposition}\label{keyresult}
Let $X$ be a Wiener functional such that $\mathcal{E}_2(X)< \infty$.
Then $M^X:=\lim_{k\rightarrow\infty}M^{k,X}$ exists weakly in $\mathrm
{H}^2$ if, and only if, the $\delta$-covariation $\langle X,B\rangle^\delta$ exists. In this case, $\langle X, B\rangle^\delta_\cdot
=[M^X,B]_\cdot$.
\end{proposition}
\begin{pf}
If $X$ has finite energy, then by Proposition~\ref{relcompact} we know
that $\{M^{k,X};k\ge1\}$ is $\sigma(\mathrm{B}^2,\mathrm
{M}^{2})$---relatively sequentially compact where all limit points
belong to $\mathrm{H}^2$. By assumption, the $\delta$-covariation
$\langle X, B\rangle^{\delta}$ exists and therefore for every $t\in[0,T]$,
\[
\lim_{i\rightarrow\infty}\bigl[M^{k_i,X},A^{k_i}\bigr]_t=
\lim_{m\rightarrow
\infty}\bigl[M^{k_m,X},A^{k_m}\bigr]_t=
\langle X,B\rangle^\delta_t\vadjust{\goodbreak}
\]
weakly in $L^1$ for any two distinct $\mathrm{B}^2$-weakly convergent
subsequences\break $\{ M^{k_i,X}\}_{i=1}^\infty$ and $\{M^{k_m,X}\}_{m=1}^\infty$. In particular, if $\lim_{i\rightarrow\infty
}M^{k_i,X}=M$ and\break $\lim_{m\rightarrow\infty}M^{k_m,X}=M'$, then Lemma
\ref{wdp} yields
\begin{eqnarray*}
\lim_{i\rightarrow\infty}\bigl[M^{k_i,X},A^{k_i}\bigr]_t&=&
\lim_{m\rightarrow
\infty}\bigl[M^{k_m,X},A^{k_m}\bigr]_t\\
&=&
\bigl[M',B\bigr]_t=[M,B]_t \qquad\mbox{weakly
in } L^1
\end{eqnarray*}
for $0\le t\le T$, and therefore $[M-M',B]_\cdot=0$. The predictable
representation of the Brownian motion yields $M=M'$. In this case,
$M^{k,X}$ should be convergent and Lemma~\ref{wdp} yields $\langle X,B
\rangle^\delta= [M^X,B]$, where $M^X:=\lim_{k\rightarrow\infty
}M^{k,X}$. Reciprocally, if $\lim_{k\rightarrow\infty}M^{k,X}=M^X\in
\mathrm{H}^2$ exists weakly in $\mathrm{B}^2$, then we may again invoke
Lemma~\ref{wdp} to conclude that $\langle X, B \rangle^\delta$ exists.
\end{pf}

The main result of this section gives the structural conditions for our
discretization scheme to work. In fact, those conditions are similar to
weak Dirichlet-type processes where the notion of covariation is
computed in terms of $\langle\cdot,\cdot\rangle^\delta$.
%
\begin{theorem}\label{mainresult}
Let $X$ be a finite energy Wiener functional such that $\lim_{k\rightarrow\infty}\delta^kX=X$ weakly in $\mathrm{B}^2$ and
$\langle X,B \rangle^\delta_\cdot$ exists. Let $(M^{k,X},N^{k,X})$ be
the canonical decomposition of $\delta^kX$. Then there exists a unique
martingale $M^X$ in $\mathrm{H}^2$ such that $N^X:=X- X_0 -
M^{X}$ satisfies the following orthogonality condition:
\[
\bigl\langle N^{X}, B\bigr\rangle^{\delta} \equiv0.
\]
If this is the case, we may
write
%
\begin{equation}
\label{mrdec} X=X_0 + M^X+ N^X
\end{equation}
and this decomposition is unique. Moreover,
$M^{k,X}\rightarrow M^X$ and $N^{k,X}\rightarrow N^X$ weakly in $\mathrm
{B}^2$ as $k\rightarrow\infty$.
\end{theorem}
\begin{pf}
By Proposition~\ref{keyresult} we know that
\[
M^X:=\lim_{k\rightarrow\infty}M^{k,X}
\]
exists and $M^X\in\mathrm{H}^2$.
From assumption
$\lim_{k\to\infty}\delta^k X= X$, we shall define
$N^X:=\lim_{k\rightarrow\infty}N^{k,X}$ weakly in $\mathrm{B}^2$. By the
very definition, we have
\[
\delta^kN^X = M^{k,X} - \delta^kM^{X}
+ N^{k,X}.
\]
The path continuity of $N^{k,X}$ yields
%
\begin{equation}
\label{fff} \bigl[\delta^kN^X, A^k
\bigr]_{t} = \bigl[M^{k,X} - \delta^kM^{X},A^k
\bigr]_t,\qquad 0\le t \le T; k\ge1.
\end{equation}

The weak convergence of $\mathbb{F}^k$ to $\mathbb{F}$ [see Lemma \ref
{smoothfiltration}, item (ii)] and relation (\ref
{martingaleprojection}) yield $\delta^kM^X=\mathbb{E}[M^X_T|\mathcal
{F}^k_\cdot]\rightarrow M^X$ uniformly in probability as $k\rightarrow
\infty$. By Lem\-ma~\ref{energyresult}, we know that $\mathcal{E}_2(M^X)<
\infty$ and therefore $\delta^kM^X\rightarrow M^X$ in $\sigma(\mathrm
{B}^2,M^2)$. Lemma~\ref{wdp}, Proposition~\ref{keyresult} and (\ref
{fff}) yield $\langle N^X, B\rangle^\delta=[M^X-M^X,B]=0$.\vadjust{\goodbreak}

The uniqueness of the decomposition is an immediate consequence of the
orthogonality property of the nonmartingale component, the predictable
representation property of the Brownian motion and the fact that
$\langle W,B\rangle^\delta= [W,B]$ for every $W\in\mathrm{H}^2$.
\end{pf}

\section{The stochastic derivative}\label{sec4}
\label{derivativesection}
In this section, we provide an explicit approximation scheme for
the martingale representation in the decomposition given in Theorem \ref
{mainresult}. The approximation will be given in terms of $\mathcal
{D}\delta^k$ which can be interpreted in the limit as a
derivative operator on the Wiener space w.r.t. Brownian motion.

For a given Wiener functional $X$, we introduce the following
family of $\mathbb{F}^k$-predictable processes:
%
\begin{equation}
\label{weightder} \mathcal{D}^{k}X:= 0 \one_{\lbra T_{0}^k, T^k_0 \rbra } + \sum
_{n=0}^{\infty}\mathcal{D}_{T^k_n}
\delta^kX \one_{\rbra T_{n}^k,T_{n+1}^k\rbra },
\end{equation}
where
\[
\mathcal{D}^k_sX=\frac{\delta^kX_{T^k_n}-
\delta^kX_{T^k_{n-1}}}{B_{T^k_n} - B_{T^k_{n-1}}} \qquad\mbox{on } \bigl\{
T_{n}^k < s \le T_{n+1}^k \bigr\}, n
\ge1.
\]
In view of Theorem~\ref{mainresult}, the goal of this section is to
show robustness of our approximation scheme in the sense that
\[
\mathcal{D}X:=\lim_{k\rightarrow\infty}\mathcal{D}^kX=H^X
\qquad\mbox{weakly}
\]
whenever $X$ satisfies the assumptions of Theorem \ref
{mainresult} such that the martingale component in (\ref{mrdec}) has a
representation $M^X=\int H^X_s\,dB_s$. One should notice that since there
is no a priori path regularity of $X$ (in particular $H^X$), one has to
choose an appropriate topology in order to get the existence of $\lim_{k\rightarrow\infty}\mathcal{D}^kX$. In the sequel, we denote by
$\lambda$ the usual Lebesgue measure on $[0,T]$.

Let us begin with the following technical lemmas. At first, the
following remark proves to be very useful for the approach taken in
this work. In fact, it will play a key role in the study of the limit
$\lim_{k\rightarrow\infty}\mathcal{D}^kX$ because it allows us to
control the quantity $(\Delta A^k_{T^k_n})^{-1}$ in (\ref{weightder}).
It is a straightforward consequence of the strong Markov property and
the $1/2$-self-similarity of the Brownian motion.
%
\begin{remark}\label{indep}
The stopping time $T^k_n-T^k_{n-1}$ is independent from $\mathcal
{G}^k_{n-1}$ for every $k,n\ge1$. Moreover, $\mathbb
{E}(T^k_n-T^k_{n-1})=2^{-2k}$ for every $k,n\ge1$.
\end{remark}
%
\begin{lemma}\label{gkn}
If $g \in L^{\infty}$, then for every $ 1 < p < \infty$,
\[
\mathbb{E}\sup_{n\ge1} \bigl|\mathbb{E}\bigl[g| \mathcal{G}^k_n
\bigr] - \mathbb {E}\bigl[g|\mathcal{G}^k_{n-1}\bigr]
\bigr|^p \one_{\{T_{n}^k\le T\}} \rightarrow 0
\qquad\mbox{as } k\rightarrow\infty.
\]
\end{lemma}
\begin{pf}
For a given $g\in L^\infty$, let $X_t=E[g|\mathcal{F}_t], 0\le t \le
T$. Recall that $\delta^kX_t = E[X_T|\mathcal{F}^k_t], 0\le t\le T$ and
therefore $\mathbb{E}[g|\mathcal{G}^k_n] = \delta^kX_{T^k_n}$ on $\{
T^k_n\le T \}$ for each $k$, $n\ge1$. Moreover, $X$ is a bounded $\mathbb
{F}$-martingale with continuous paths. Remark~\ref{ucpconv} yields
$\delta^k X\rightarrow X$ strongly in $\mathrm{B}^p(\mathbb{F})$ as
$k\rightarrow\infty$, $p>1$. We shall write
%
\begin{eqnarray}
\label{a17}
&&
\mathbb{E}^{1/p}\sup_{n\ge1}\bigl|\mathbb{E}\bigl[g|
\mathcal {G}^k_n\bigr]-\mathbb{E}\bigl[g|
\mathcal{G}^k_{n-1}\bigr]\bigr|^p
\one_{\{ T_{n}^k\le T \}} \nonumber\\
&&\qquad\le \mathbb{E}^{1/p}\sup_{n\ge
1}\bigl|
\delta^kX_{T^k_n} - X_{T^k_n}\bigr|^p
\one_{\{T_{n}^k\le T\}}
\nonumber\\[-8pt]\\[-8pt]
&&\qquad\quad{} + \mathbb{E}^{1/p}\sup_{n\ge
1}|X_{T^k_n}-X_{T^k_{n-1}}|^p
\one_{\{T_{n}^k\le T\}}
\nonumber\\
&&\qquad\quad{}+
\mathbb{E}^{1/p}\sup_{n\ge1}\bigl|\delta^kX_{T^k_{n-1}}-X_{T^k_{n-1}}\bigr|^p
\one_{\{T_{n}^k\le T\}}. \nonumber
\end{eqnarray}

The first\vspace*{2pt} and last terms in (\ref{a17}) vanish. Moreover, $\lim_{k\rightarrow\infty}\mathbb{E}^{1/p}\sup_{n\ge
1}|X_{T^k_n}-X_{T^k_{n-1}}|^p\one_{\{T_{n}^k\le T\}}=0$ because of the
path continuity of $X$ together with the fact that $\sup_{n\ge
1}|T^k_n-T^k_{n-1}|\rightarrow0$ a.s. as $k\rightarrow\infty$.
\end{pf}
Now we are in position to prove the existence of the stochastic derivative.

\begin{theorem}\label{mainint}
Let $X$ be a Wiener functional satisfying the assumptions of Theorem
\ref{mainresult} with the weak decomposition represented by
%
\begin{equation}
\label{klj} X=X_0 + \int H^X\,dB_s +
N^X
\end{equation}
for an adapted process $H^X$ in $L^2(\lambda\times\mathbb
{P})$ and $\langle N^X,B \rangle^\delta=0$. Then $H^X$ can be
approximated by the $L^2(\lambda\times\mathbb{P})$-weak limit $\mathcal
{D}X=\lim_{k\rightarrow\infty}\mathcal{D}^kX=H^X$.
\end{theorem}
\begin{pf}
The unique orthogonal decomposition (\ref{klj}) represented by an
adapt\-ed process $H^X$ is a consequence of Theorem~\ref{mainresult}
together with the martingale representation of the Brownian motion.
Therefore, it only remains to prove the existence of $\mathcal{D}X$.
For this, let us consider a finite energy Wiener functional $X$ and let
us fix $0\le t\le T$ and $g\in L^\infty$. In order\vspace*{1pt} to shorten notation,
let us write $\xi^k_n:=(T^k_n-T^k_{n-1})\one_{\{T^k_n \le T\}}$,
$g^k_n:=\mathbb{E}[g|\mathcal{G}^k_n]-\mathbb{E}[g|\mathcal
{G}^k_{n-1}]$ for $k,n\ge1$ and\vspace*{2pt} $C$ is a constant which may differ
from line to line. By the very definition, for every $k\ge1$ and $t>0$,
%
\begin{eqnarray}
\label{first1}
g\int_0^t\mathcal{D}_s^kX\,ds
&=& g\sum_{n=1}^\infty\mathcal
{D}_{T^k_{n-1}}\delta^kX\xi^k_n
\one_{\{T_{n-1}^k\le t\}}\nonumber\\[-8pt]\\[-8pt]
&&{} - g\sum_{n=1}^\infty
\mathcal{D}_{T^k_{n-1}}\delta^kX\bigl(T^k_n-t
\bigr)\one_{\{
T_{n-1}^k< t\le T^k_n\}}.\nonumber
\end{eqnarray}
At first, a simple application of Remark~\ref{indep} and the very
definition of $\mathcal{D}^kX$ yield
%
\begin{equation}
\label{v1}\qquad \mathbb{E}\int_0^T\bigl|
\mathcal{D}_s^kX\bigr|^2\,ds = \mathbb{E}\sum
_{n=1}^\infty \bigl|\mathcal{D}_{T^k_{n-1}}
\delta^kX\bigr|^2\xi^k_n
\one_{\{T_{n-1}^k\le T\}
}\le\mathcal{E}_2(X),\qquad k\ge1.
\end{equation}
By H\"{o}lder inequality and (\ref{v1}), the second term in (\ref
{first1}) vanishes as follows:
\begin{eqnarray*}
&&
\mathbb{E}\sum_{n=1}^\infty\bigl|g
\mathcal{D}_{T^k_{n-1}}\delta^kX\bigl(T^k_n-t
\bigr)\bigr|\one_{\{T_{n-1}^k< t\le T^k_n\}} \\[-1pt]
&&\qquad\le C\mathbb{E}\sum_{n=1}^\infty\bigl|
\mathcal{D}_{T^k_{n-1}}\delta^kX\bigr|\xi^k_n
\one_{\{
T_{n-1}^k < t \le T^k_n\}}
\\[-1pt]
&&\qquad\le C\mathcal{E}^{1/2}_2(X)\times\mathbb{E}^{1/2}
\sup_{n\ge1}\bigl|\xi^k_n\bigr|\one_{\{T_{n}^k\le T\}}
\rightarrow 0 \qquad\mbox{as } k\rightarrow\infty.
\end{eqnarray*}
By using Remark~\ref{indep}, we shall write
%
\begin{eqnarray}
\label{v2} \quad\mathbb{E} g\sum_{n=1}^\infty
\mathcal{D}_{T^k_{n-1}}\delta^kX\xi^k_n
\one_{\{T_{n-1}^k\le t\}}&=& \mathbb{E}\sum_{n=1}^\infty
g^k_n\mathcal{D}_{T^k_{n-1}}\delta^kX
\xi^k_n \one_{\{
T_{n-1}^k\le t\}}
\nonumber\\[-8pt]\\[-8pt]
&&{}+ \mathbb{E} g\sum_{n=1}^\infty\Delta
\delta^kX_{T^k_{n-1}}\Delta A^k_{T^k_{n-1}}
\one_{\{T_{n-1}^k\le t\}}.
\nonumber
\end{eqnarray}
The first term in (\ref{v2}) vanishes as follows. By applying Lemma \ref
{gkn}, (\ref{v1}) and H\"{o}lder inequality we have
\begin{eqnarray*}
\mathbb{E}\sum_{n=1}^\infty
\bigl|g^k_n\mathcal{D}_{T^k_{n-1}}\delta^kX
\xi^k_n \bigr|\one_{\{T_{n-1}^k\le t\}}&\le&\mathcal
{E}^{1/2}_2(X)\mathbb{E}^{1/2}\sup_{n\ge1}\bigl|g^k_n\bigr|^2
\sum_{n=1}^\infty\xi^k_n
\one_{\{T_{n-1}^k\le t\}}
\\
&\rightarrow& 0 \qquad\mbox{as } k\rightarrow\infty.
\end{eqnarray*}
Summing up the above arguments, we arrive at the following conclusion:
%
\begin{equation}
\label{fidederi1} \lim_{k\rightarrow\infty}\mathbb{E}g\int_0^t
\mathcal{D}^k_sX\,ds \mbox{ exists}\quad\mbox{if, and only if}\quad
\lim_{k\rightarrow\infty}\mathbb{E}g\bigl[\delta^kX,A^k
\bigr]_t \mbox{ exists}.\hspace*{-28pt}
\end{equation}
In other words, the estimate (\ref{v1}), the existence of $\langle X,
B\rangle^\delta$ and (\ref{fidederi1}) allow us to conclude that
\[
\lim_{k\rightarrow\infty}\mathcal{D}^kX \qquad\mbox{exists weakly in }
L^2(\lambda\times\mathbb{P}).
\]
It follows from the above steps that
%
\begin{eqnarray}
\label{fidederi2} \lim_{k\rightarrow\infty}\mathbb{E}g\int_0^t
\mathcal {D}^k_sX\,ds&=&\mathbb{E}g\langle X, B
\rangle^\delta_t=\mathbb {E}g\bigl[M^X,B
\bigr]_t
\nonumber\\[-8pt]\\[-8pt]
&=&\mathbb{E}g\int_0^tH^X_s\,ds,\qquad
0\le t\le T, g\in L^\infty,
\nonumber
\end{eqnarray}
where the martingale component is represented by a
progressive process $H^X$ in $L^2(\lambda\times\mathbb{P})$, that is,
$M^X_t=\int_0^tH^X_s\,dB_s$ for $0\le t\le T$. Identity (\ref{fidederi2})
shows that $\mathcal{D}^k X\rightarrow H^X$ weakly in $L^2(\lambda
\times\mathbb{P})$ as $k\rightarrow\infty$. The proof of the theorem
is complete.
\end{pf}
%
\begin{remark}\label{mainint3}
Under assumption $\mathcal{E}_2(X)< \infty$, relations
(\ref{fidederi1}) and (\ref{fidederi2}) allow us to conclude that $\lim_{k\rightarrow\infty}\mathcal{D}^{k}X$ exists weakly in $L^2(\lambda
\times\mathbb{P})$ if, and only if, the $\delta$-covariation $\langle
X, B\rangle^\delta$ exists. By Proposition~\ref{keyresult}, $\lim_{k\rightarrow\infty}\mathcal{D}^{k}X$
exists if, and only if, $\lim_{k\rightarrow\infty}M^{k,X}$ exists which shows a strong robustness
of our approximation scheme. In this case, $\lim_{k\rightarrow\infty
}\mathcal{D}^kX=H^X$ where $M^X=\int H^X_s\,dB_s=\lim_{k\rightarrow\infty
}M^{k,X}$.
\end{remark}

By applying Theorem~\ref{mainint} to the classical It\^o representation
theorem we arrive at the following result.
%
\begin{corollary}\label{clarkformula}
If $F$ is an $\mathcal{F}_T$-square integrable random variable, then
\[
F = \mathbb{E}[F] + \int_0^T
\mathcal{D}_sF\,dB_s,
\]
where
%
\begin{equation}
\label{representation} \mathcal{D}F =\lim_{k\rightarrow\infty} \sum
_{n=1}^\infty\frac{\mathbb
{E}[F|\mathcal{G}^k_n]
- \mathbb{E}[F|\mathcal{G}^k_{n-1}]}{B_{T^k_n}-B_{T^k_{n-1}}}\one_{
\rbra T_{n}^k, T^k_{n+1}\rbra }
\end{equation}
weakly in $L^2(\lambda\times\mathbb{P})$.
\end{corollary}
\begin{pf}
If\vspace*{1pt} $X_t=\mathbb{E}[F|\mathcal{F}_t]$, $0\le t\le T$, then $\delta^kX_t=\mathbb{E}[F] + \oint_0^t\mathcal{D}_s\delta^kX\,dA^k_s$ where
$\delta^kX_{T^k_n}=\mathbb{E}[F|\mathcal{G}^k_n]; k,n\ge1$. Since\vspace*{1pt} $X$
is a square-integrable martingale, a~simple application of Theorem \ref
{mainint} yields (\ref{representation}).
\end{pf}
%
\begin{remark}
Corollary~\ref{clarkformula} and\vspace*{1pt} the classical Clark--Ocone formula
yield $\mathcal{D}_tF = \mathbb{E}[D_tF|\mathcal{F}_t]$ where $D$
denotes the Gross--Sobolev derivative of $F$ in $L^2(\mathbb{P})$. If
$F$ is not differentiable in the sense of Malliavin calculus, the
Gross--Sobolev derivative $D_tF$ is interpreted as a generalized
process where $\mathbb{E}[D_tF|\mathcal{F}_t]$ can be interpreted as a
real-valued process in $L^2(\lambda\times\mathbb{P})$ (see, e.g.,~\cite{bermin} for more details).
\end{remark}

\section{The Clark--Ocone formula algorithm}\label{sec5}\label{numericalsection}
In this section, we illustrate the theory developed in this article
with some numerical examples. The goal here is to show that our
approximation scheme can be easily implementable where a step-by-step
algorithm for the Clark--Ocone formula is presented. We illustrate the
method with the problem of hedging contingent claims in a complete
market. For simplicity of exposition and comparison with exact known
formulas, we will work on a simple diffusion setup together with
well-known types of contingent claims. We stress that the algorithm
presented in Section~\ref{ALG} holds for any square integrable $\mathcal
{F}_T$-random variable.

In this section, the market consists of two assets: one riskless asset
$S^0$ and one risky asset $S$ with continuous paths. We will specify
the evolution of the assets directly under the unique equivalent
martingale measure $\mathcal{Q}$ together with the respective $\mathcal
{Q}$-Brownian motion $W$. We assume that they are given by
%
\begin{equation}
\label{risky} dS^0_t=rS^0_t\,dt,\qquad
S^0_0=1;\qquad dS_t = rS_t\,dt + \sigma
S_t\,dW_t,
\end{equation}
where $\sigma>0$ and $r>0$. It is well known (see, e.g.,
\cite{karatzas,bermin}) that for any given contingent claim $F\in
L^2(\mathcal{Q}, \mathcal{F}_T)$, the correspondent replicating
strategy $\theta$ is derived by the Clark--Ocone--Karatzas
\cite{karatzas} formula as
%
\begin{equation}
\label{theta} \theta_t= e^{-r(T-t)} \sigma^{-1}(S_t)^{-1}F^o_t,
\end{equation}
where $F^o_\cdot:=\mathbb{E}_{\mathcal{Q}}[D_\cdot F|\mathcal
{F}_\cdot]$ and $D$ is the Gross--Sobolev derivative. Since the
filtration generated by $W$ coincides with $\mathbb{F}$, Theorem \ref
{clarkformula} still holds under the correspondent martingale measure
$\mathcal{Q}$ as well. In the sequel, with a slight abuse of notation,
we also denote by $\mathbb{E}$ the expectation under the measure~$\mathcal{Q}$.

%
\begin{remark}
In Fournie et al.~\cite{FOURNIE1} and also in a series of
papers~\cite{CVI,elie,ben,higa}, the idea is to express the
optional projection $F^o$ by $\mathbb{E}[F. G| \mathcal{F}_\cdot]$ for
a suitable random variable $G$ which in general is represented by a
Skorohod integral. In this case, a smooth underlying Markovian
structure plays a key role. In this work, we take a rather different
strategy which is fully based on the information generated by the
stopping times $(T^k_n)_{k,n\ge1}$ which allows us to treat any
$L^2(\mathcal{F}_T)$-random variable (see also Remark~\ref{sspositive}).
\end{remark}

To illustrate our method, we will study three types of derivatives: a
European call option, a digital option and a barrier option given,
respectively, by
%
\begin{equation}
\label{payoffs} \max\{S_T-K, 0\},\qquad \one_{\{ S_T \leq K \}},\qquad
\one_{\{ M^{S}_{0,T} \leq
K \}},
\end{equation}
where\vspace*{1pt} $M^{S}_{0,T}:=\sup_{0\le t\le T}|S_t|$. It is well
known that for these types of claims, there exist closed formulas for
hedging (see, e.g.,~\cite{bermin}, examples 4.1 and 5.3).

\subsection{The algorithm}\label{sec5.1}\label{ALG}
The method is fully based on the space-filtration discretization scheme
induced by the stopping times $\{T^k_n; k,n\ge1 \}$. In the sequel, we
fix an $L^2(\mathcal{F}_T)$-random variable $F$ and our goal is to
describe an algorithm to calculate the optional projection $F^o_0$
which yields the hedging $\theta_t$ at time $t=0$. The other times can
be recovered from this case by a standard shift argument. From\vadjust{\goodbreak} (\ref
{representation}), it follows that for sufficiently small $\varepsilon
>0$ and $k$ large enough,
%
\begin{equation}
\label{F0ap} \frac{1}{\varepsilon}\mathbb{E} \int_{0}^{\varepsilon}
\mathcal {D}^k_sF\,ds = \frac{1}{\varepsilon}\mathbb{E} \int
_{T^k_1}^{\varepsilon
}\mathcal{D}^k_sF\,ds
\sim\frac{1}{\varepsilon}\mathbb{E}\int_{0}^{\varepsilon}F^o_s\,ds
\sim\mathbb{E}F^o_0
\end{equation}
as long as $0$ is a Lebesgue point of $t\mapsto\mathbb
{E}F^o_{t}$. For the purpose of hedging, we may assume that this is the
case. Otherwise, we shall always find a point in a neighborhood of
$t=0$ such that (\ref{F0ap}) holds. In order to speed up the
convergence of the algorithm we take
\[
\mathbb{E}\frac{1}{\varepsilon- T^k_1} \int_{T^k_1}^{\varepsilon
}
\mathcal{D}^k_sF\,ds,
\]
instead of $\frac{1}{\varepsilon}\mathbb{E} \int_{T^k_1}^{\varepsilon}\mathcal{D}^k_sF\,ds$ in (\ref{F0ap}). One should
notice that $\frac{1}{\varepsilon-T^k_1}\one_{\{T^k_1 < \varepsilon\}
}\rightarrow\frac{1}{\varepsilon}$ in $L^2$ as $k\rightarrow\infty$
and therefore
%
\begin{equation}
\label{F0ap1} \mathbb{E}\frac{1}{\varepsilon- T^k_1} \int_{T^k_1}^{\varepsilon
}
\mathcal{D}^k_sF\,ds \sim\mathbb{E}F^o_0
\end{equation}
for $k$ sufficiently large. By the very definition,
%
\begin{eqnarray}
\label{F0ap2} \frac{1}{\varepsilon- T^k_1} \int_{T^k_1}^{\varepsilon}
\mathcal {D}^k_sF\,ds &=& \frac{1}{\varepsilon- T^k_1} \Biggl[\sum
_{n=1}^{\infty} \mathcal {D}^k_{T_{n}^k}
F\bigl(T_{n+1}^k - T_{n}^k\bigr)
\one_{ \{T_{n+1}^k \leq
\varepsilon\} }
\nonumber\\[-8pt]\\[-8pt]
&&\hspace*{35pt}{} + \sum_{n=1}^{\infty} \mathcal{D}^k_{T_{n}^k}
F \bigl(\varepsilon- T_{n}^k\bigr) \one_{ \{ T_{n}^k < \varepsilon\leq
T_{n+1}^k\} }
\Biggr].
\nonumber
\end{eqnarray}
The whole structure of the algorithm is based on the perfect simulation
of the first passage times $\{T^k_n-T^k_{n-1};k,n\ge1\}$. Based on the
density of $T^k_1$, Burq and Jones~\cite{burq} proposes a very simple
and efficient algorithm. We refer the reader to this work for a
detailed exposition of the perfect simulation method for the stopping
times.\vspace*{8pt}

(S1) \textit{Simulation of $\{A^k;k\ge1 \}$}.

\begin{itemize}
\item One chooses $k\ge1$ which represents the discrimination level of
the Brownian motion.

\item One generates the stopping times $\{T_{n}^k-T^k_{n-1}; n\ge1\}$
according to the algorithm described by~\cite{burq}.

\item One simulates the family $\{\sigma^k_n;n\ge1 \}$ independently
from \mbox{$\{T^k_n-T^k_{n-1};n\ge1\}$}. The i.i.d. family $\{\sigma^k_n;n\ge
1 \}$ must be simulated according to the Bernoulli random variable
$\sigma^k_1$ such that $\mathbb{P}[\sigma_1^k = i]=1/2$ for $i=-1,1$.
This simulates the jump process $A^k$.
\end{itemize}
The next step is the simulation of $\mathcal{D}^kF$ where the
conditional expectations $\{\mathbb{E}[F|\mathcal{G}^k_n]; n,k\ge1\}$
play a key role.\vspace*{8pt}

(S2) \textit{Simulation of the stochastic derivative}.

\begin{itemize}
\item Fix a small $\varepsilon> 0$.\vadjust{\goodbreak}

\item Generate one sample of $A^k$ according to (S1) for a
large $k\ge1$. From this sample, one takes $(t_{1}^k, \sigma^k_1);
\ldots; (t_{n}^k, \sigma^k_n)$ such that $t_{n}^k < \varepsilon\leq
t_{n+1}^k$.\vspace*{1pt}

\item For each $0 \leq j \leq n$, one applies Monte Carlo simulation to
obtain an approximation of $\mathbb{E}[F|(t_{1}^k, \sigma^k_1), \ldots,
(t_{j}^k, \sigma^k_j)]$ (see Remark~\ref{sspositive}). This object is
denoted by $\hat{\mathbb{E}}[F|(t_{1}^k, \sigma^k_1), \ldots, (t_{j}^k,
\sigma^k_j)]$.
\end{itemize}
Therefore, an approximation for the stochastic derivative $\mathcal
{D}^kF$ is given by
\[
\hat{\mathcal{D}}^k_{t_{j}^k} F:= \frac{\hat{\mathbb{E}}[F|(t_{1}^k,
\sigma^k_1), \ldots, (t_{j}^k, \sigma^k_j)]- \hat{\mathbb
{E}}[F|(t_{1}^k, \sigma^k_1), \ldots, (t_{j-1}^k, \sigma^k_{j-1})]}{2^{-k} \sigma^k_j}
\]
for $1\le j\le n$. Then one can define the following object according
to (\ref{F0ap1}) and (\ref{F0ap2}):
\[
\hat{F}^o_0(\varepsilon,k):= \frac{1}{\varepsilon- t^k_1} \Biggl[
\sum_{j=1}^{n-1} \hat{\mathcal{D}}^k_{t_{j}^k}
F\bigl(t_{j+1}^k - t_{j}^k\bigr) +
\hat {\mathcal{D}}^k_{t_{n}^k} F \bigl(\varepsilon-
t_{n}^k\bigr) \Biggr].
\]
\begin{itemize}
\item From (\ref{F0ap1}) and (\ref{theta}), the correspondent
replicating strategy for this path can be approximated by
\[
\hat{\theta}_0(\varepsilon,k):= e^{-r(T)}
\sigma^{-1}(S_0)^{-1} \hat {F}^o_0(
\varepsilon,k).
\]
\item Repeat these steps several times and take the mean of the
strategies $\hat{\theta}_0(\varepsilon,k)$ as the estimative for the
replicating strategy $\theta_t$ at the initial point $t=0$.
\end{itemize}
%
\begin{remark}\label{sspositive}
The methodology presented in this section is rather general in the
sense that the only assumption that is made is the possibility to
simulate the expectation ${\mathbb{E}}[F|(t_{1}^k, \sigma^k_1), \ldots,
(t_{j}^k, \sigma^k_j)]$ by a Monte\vspace*{1pt} Carlo method. In the classical
Black--Scholes setup, one can simulate it by means of random samples
generated by
%
\begin{equation}
\label{bsequation} S_t^k= \exp \bigl[ \bigl(r-
\tfrac{1}{2}\sigma^2\bigr)t + \sigma A^k_t
\bigr],\qquad 0\leq t \leq T.
\end{equation}
In the general positive semimartingale case, the expectation can be
simulated based on the Euler--Maruyama method, for example, for
\[
dS_t=r_tS_t\,dt+\sigma_tS_t\,dW_t;\qquad
0\le t\le T,
\]
with the increments $(W_{T^k_n}-W_{T^k_{n-1}})$ computed in terms of
$(T^k_n)$ and $\sigma$ can be random.
\end{remark}

\subsection{Numerical examples}\label{sec5.2}
As a simple illustration of our method, we consider three types of
derivatives: a European call option, a digital option and a barrier
option. The stock price is $49$, the risk-free interest rate is $5\%$
per annum, the stock price volatility is $20 \%$ per annum, the time to
maturity is $20$ weeks ($0.3846$ years) and the expected return from
the stock is $13 \%$ per annum. We use strike price $K=50$ for the\vadjust{\goodbreak}
European call option and digital option and $K=55$ for the barrier
option. In order to develop the simulation process we choose a
discrimination level of order $k=4$. The main point of the algorithm is
the approximation of the conditional expectations by Monte Carlo
simulation. For each case, we generate 10,000 paths of $A^k$, we
evaluate the payoff function on each path based on (\ref{bsequation})
and we take the mean as the estimative of the conditional expectation.
We choose $\varepsilon=0.02$ and we generate $1000$ samples of $A^k$
stopped at $0.02$.

For each sample, we calculate the respective hedging value at time
$t=0$. The estimative of the hedging at time $t=0$ is given by the mean
of the hedging values from the correspondent samples (see Figure \ref
{options}). The $\% \mathit{Error}$ is the absolute value of the difference
between estimated value and exact value divided by the exact value.
With $1000$ paths, we obtain an error of $0.57 \%$ for the European
call option, $0.48 \%$ for the digital option and $0.1 \%$ for the
barrier option.

\begin{sidewaysfigure}

\includegraphics{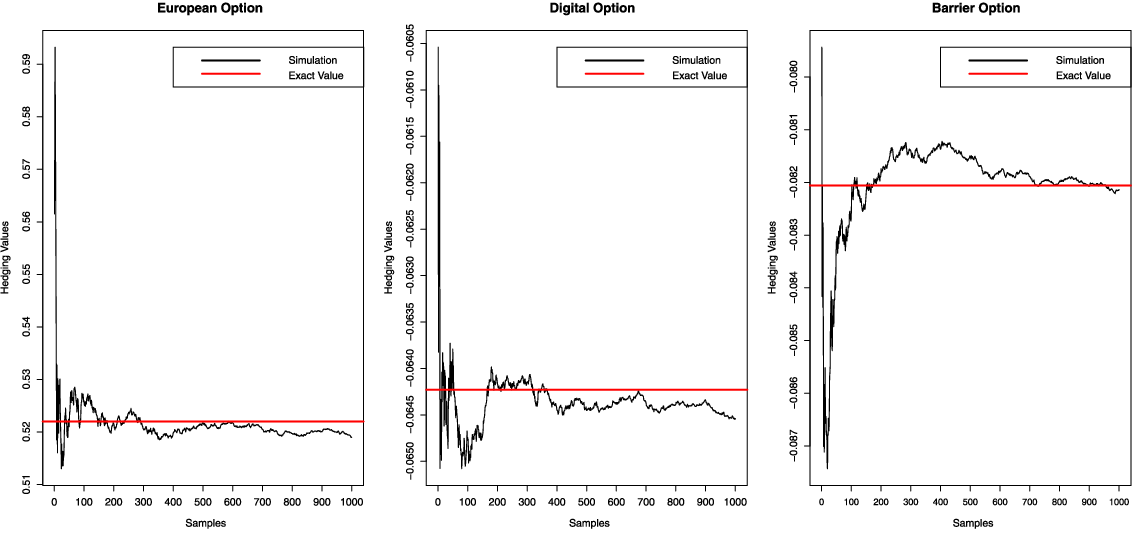}

\caption{Monte Carlo simulations for hedging.}\label{options}
\end{sidewaysfigure}

At this stage, one can say that the method proposed in this work is
rather general compared with the more classical ones based on the
existence of densities (see, e.g.,~\cite{higa} and other references
therein). Moreover, it does not require further smoothness assumptions
and it can be easily implementable without requiring advanced
mathematical calculations as in the classical literature in
mathematical finance.

\section{Optimal stopping}\label{sec6}
In this section, we illustrate the techniques developed in this paper
with the optimal stopping time problem based on a Wiener
functional~$X$. Throughout this section we fix a bounded positive
Wiener functional $X$ with continuous paths. To shorten notation, in
this section we shall extend the time domain of all stochastic
processes $X$ to $[0, \infty)$ as follows: $X_t = X_T$ for every $t
\geq T$. In the sequel, we denote the set of all $\mathbb{F}$-stopping
times by $\operatorname{ST}(\mathbb{F})$.
%
\begin{definition}
For a fixed $\varepsilon> 0$, we say that the stopping time $\tau^{\varepsilon}$ is $\varepsilon$-optimal if
\[
\mathbb{E} X_{\tau^{\varepsilon}} \geq\sup_{\tau\in \operatorname{ST}({\mathbb{F}})} \mathbb{E}
X_{\tau} - \varepsilon.
\]
\end{definition}

From Remark~\ref{ucpconv}, we know that $\lim_{k\rightarrow\infty
}\delta^kX=X$ strongly in $\mathrm{B}^2$ and therefore,
%
\begin{equation}
\label{supconv} \sup_{\tau\in \operatorname{ST}(\mathbb{F})} \mathbb{E}\delta^kX_{\tau}
\rightarrow\sup_{\tau\in \operatorname{ST}(\mathbb{F})} \mathbb{E} X_{\tau}
\qquad\mbox{as } k
\rightarrow \infty.
\end{equation}

Let $\mathcal{U}^k$ be the class of totally inaccessible $\mathbb
{F}^k$-stopping times. It is well known (see, e.g.,~\cite{yan}) that $S
\in\mathcal{U}^k$ if, and only if, $S= \sum_{n=1}^{\infty} T^k_{n}
\one_{ \{S=T^k_n\}}$. The fact that $\delta^k X $ is constant on the
stochastic interval $\lbra T_{n}^k, T_{n+1}^k\lbra $ for each $n \geq0$ and
convergence (\ref{supconv}) allow us to state the following result.

\begin{lemma}For each $\varepsilon> 0 $, we have
\[
\sup_{\tau\in \operatorname{ST}(\mathbb{F})} \mathbb{E} X_{\tau} \leq\sup_{\tau^k \in
\mathcal{U}^k}
\mathbb{E} \delta^k X_{\tau^k} + \varepsilon
\]
for every $k$ sufficiently large.
\end{lemma}
%
\begin{definition}
For a given $\varepsilon\ge0$, we say that $\tau^{k,\star} (\in
\mathcal{U}^k)$ is $(k, \varepsilon)$-optimal, if
\[
\mathbb{E} \delta^k X_{\tau^{k, \star}} \ge\sup_{\tau^k \in\mathcal
{U}^k}
\mathbb{E} \delta^k X_{\tau^k} - \varepsilon.
\]
\end{definition}
Summing up the above results together with Remark~\ref{ucpconv}, the
following proposition holds.
%
\begin{proposition}\label{optimalst}
Let $X$ be a bounded positive continuous Wiener functional. For a given
$\varepsilon>0$, each $(k,0)$-optimal stopping time $\tau^{k,\star}$
is $\varepsilon$-optimal for every $k$ sufficiently large.
\end{proposition}

\subsection{A dynamic programming principle}\label{sec6.1}\label{dpp}
In the sequel, we provide a dynamic programming principle to
approximate a $(k, 0)$-optimal stopping time for any Wiener functional
$X$ satisfying the assumptions of Proposition~\ref{optimalst}. Let
$S^k$ be the Snell envelope\vspace*{1pt} of $\delta^kX$, that is, the
minimal positive $\mathbb{F}^k$-supermartingale which dominates
$\delta^k X$. The dynamic programming principle can be written as
follows. For a fixed $\omega\in\{T_{n}^k \leq T < T_{n+1}^k\}$, we
shall write
\[
\cases{ S^k_{T_{n}^k} (\omega) = \delta^k
X_{T_{n}^k} (\omega),
\cr
S^k_{T_{j}^k } (\omega) = \max
\bigl\{\delta^k X_{T_{j}^k } (\omega); \mathbb{E} \bigl[
S^k_{T_{j+1}^k} \mid\mathcal{G}^k_j
\bigr](\omega) \bigr\}, &\quad $j \leq n$.}
\]
By a backward induction argument, the dynamic programming principle can
be written in terms of optimal stopping
times as
%
\begin{equation}
\label{algppd} \cases{ \tau^{k,\star}_n (\omega):=
T_{n}^k (\omega),
\vspace*{1pt}\cr
\tau^{k,\star}_{n-1} (
\omega):= T_{n-1}^k (\omega) \one_{ G^{k}_{n-1}} (\omega) +
\tau^{k,\star}_{n}(\omega) \one_{  (G^{k}_{n-1}
)^c} (\omega),
\cr
\tau^{k,\star}_{j} (\omega):= T_{j}^k (
\omega) \one_{ G^{k}_j} (\omega ) + \tau^{k,\star}_{j+1} (
\omega) \one_{  (G^{k}_j  )^c} (\omega);&\quad $\omega\in\Omega$,}
\end{equation}
where
%
\begin{equation}
\label{algppd1} G^{k}_j:= \bigl\{\delta^k
X_{T_{j}^k} \geq\mathbb{E} \bigl[\delta^k X_{\tau^{k,\star}_{j+1}}
\mid\mathcal{G}^k_j \bigr] \bigr\},\qquad j\le n-1.
\end{equation}
In this case, the $\tau^{k,\star}_0$ is $(k,0)$-optimal and the value
function is given by
%
\begin{equation}
\label{algppd2} \mathbb{E} S^k_0 = \mathbb{E}
\delta^k X_{\tau^{k,\star}_0}.
\end{equation}

\subsection{Non-Markovian examples}\label{sec6.2} We shall consider a significant
class of non-Markovian examples which fits into the assumptions of
Proposition~\ref{optimalst}. For instance, for a given bounded
continuous function\vadjust{\goodbreak} $f\dvtx\mathbb{R}\rightarrow\mathbb{R}_+$, let us
consider the following Wiener functionals:
%
\begin{equation}
\label{nexamples} f\bigl(B^H\bigr) \qquad\mbox{for } H \in(0,1),
\end{equation}
where $B^H$ is the fractional Brownian motion with parameter $H\in
(0,1)$. Based on the simulation of $\{(T^k_n-T^k_{n-1}, \sigma^k_n);n\ge1 \}$ [see
(S1) in Section~\ref{numericalsection}]
and the dynamic programming principle for $X=f(B^H)$ in the last
section, it is straightforward to develop an algorithm to approximate a
$(k,0)$-optimal stopping time. In the sequel, the index $\ell\in\{
1,\ldots, N \}$ encodes the $\ell$th iteration in a given\vspace*{1pt} dynamic
programming procedure and $\{A^{k,\ell}\}_{\ell=1}^N$ are independent
copies of $A^k$.\vspace*{8pt}

(A1) \textit{Dynamic programming algorithm.}

\begin{itemize}
\item Fix a large $k\ge1$ and generate one sample from $A^{k,\ell}$
based on (S1) and take $(t^{k,\ell}_1,
\sigma^{k,\ell}_1),\ldots, (t^{k,\ell}_n,\sigma^{k,\ell}_n)$ such that
$t^{k,\ell}_n\le T <
t^{k,\ell}_{n+1}$.

\item One sets $\tau^{k,\ell, \star}_{n}=t^{k,\ell}_n$.

\item One proceeds backward by taking the time $\tau^{k,\ell, \star
}_{j-1}$ given by
\[
\hspace*{-6pt}\tau^{k,\ell, \star}_{j-1}= \cases{ t^{k,\ell}_{j-1}, &\quad
if $f \bigl(A^{k,\ell,H}_{t_{j-1}^{k,\ell}} \bigr) \geq\mathbb{E} \bigl[ f
\bigl( A^{k,H}_{\tau^{k,\ell, \star}_j} \bigr) \mid \bigl(t^{k,\ell}_1,
\sigma^{k,\ell}_1\bigr),\ldots, \bigl(t^{k,\ell}_{j-1},
\sigma^{k,\ell}_{j-1}\bigr) \bigr]$,
\vspace*{2pt}\cr
\tau^{k,\ell,\star}_{j},
&\quad if $f \bigl(A^{k,\ell,H}_{t_{j-1}^{k,\ell}} \bigr) < \mathbb{E} \bigl[ f
\bigl( A^{k,H}_{\tau^{k,\ell,\star}_j} \bigr) \mid\bigl(t^{k,\ell
}_1,
\sigma^{k,\ell}_1\bigr),\ldots, \bigl(t^{k,\ell}_{j-1},
\sigma^{k,\ell}_{j-1}\bigr) \bigr]$,}\hspace*{-4pt}
\]
\end{itemize}
for any $ j \leq n$. The value $A^{k,\ell,H}_t$ is obtained from
$A^{k,\ell}$ via the Volterra representation of the fractional Brownian
motion as
%
\begin{equation}
\label{volterraS} A^{k,\ell,H}_{t}:= \int_0^{t}
K(t,s)\,dA^{k,\ell}_s,\qquad 0\le t\le T, H \in(0,1),
\end{equation}
for a suitable square-integrable kernel $K(t,s)$ (see, e.g.,
\cite{hu}). The conditional expectations
\[
\mathbb{E} \bigl[f \bigl( A^{k,H}_{\tau^{k,\ell,\star}_j} \bigr) \mid
\bigl(t^{k,\ell}_1,\sigma^{k,\ell}_1\bigr),
\ldots, \bigl(t^{k,\ell}_{j-1}, \sigma^{k,\ell}_{j-1}
\bigr) \bigr],\qquad j\le n,
\]
are approximated by Monte Carlo methods via simulation of $A^{k,\ell}$
described in (S1), Section~\ref{ALG} and (\ref{volterraS}).
\begin{itemize}
\item One repeats the previous steps several times, $\ell=1, \ldots,
N$ and the optimal value function (\ref{algppd2}) is approximated by $
\frac{1}{N} \sum_{\ell=1}^N f  ( A^{k,\ell,H}_{\tau^{k,\ell,\star
}_0}  )$ for large $N$.
\end{itemize}
The above formulation in terms of stopping rules (rather than in terms
of value functions) is essential to our approach as well as in other
probabilistic methods based on discretizations of the Snell envelope.
The main feature of this methodology is the computation of conditional
expectations which is generically based on the following alternatives:
projections on $L^2(\mathbb{P})$ (see, e.g.,~\cite{Longstaff}),
quantization (as in~\cite{bally}) and representation formulas based on
Malliavin calculus (see, e.g.,~\cite{touzi}).\vadjust{\goodbreak} An important drawback of
all these methodologies is that they essentially rely on an induced
Markov chain arising from a time-discretization scheme of a
continuous-time Markov process. Dynamic programming methods are not
directly usable in genuinely non-Markovian cases due to the nontrivial
time-correlation generated by the driving noise.

We circumvent this problem by introducing a space-filtration
discretization scheme which allows us to write the original optimal
stopping problem in terms of the information flow $(\mathcal
{G}^k_n)_{k,n\ge1}$. Conditional expectations appearing in the dynamic
programming principle (\ref{algppd})--(\ref{algppd2}) can now be fairly
simulated since most examples of interest can be viewed in terms of the
process $A^k$ as explained in Remark~\ref{sspositive}. A numerical
study is needed in order to precisely evaluate our method with the more
classical approaches, a topic which will be further explored in a
forthcoming paper.

\section*{Acknowledgments}

We are grateful for helpful discussions with Pedro Catuogno, Francesco
Russo and Frederi Viens.



\printaddresses

\end{document}